\documentclass[a4paper,reqno,11pt]{amsart}
\pdfoutput=1
\usepackage[american]{babel}

\usepackage{mathpazo}
\usepackage[T1]{eulervm}
\usepackage[scaled]{helvet}
\usepackage[a4paper, scale={0.72,0.74}, marginratio={1:1}, footskip=7mm, headsep=10mm]{geometry}

\usepackage{indentfirst}

\usepackage{emptypage}  

\usepackage{perpage}

\usepackage{enumitem}

\usepackage{fixltx2e}

\usepackage{relsize}

\usepackage[dvipsnames]{xcolor}

\usepackage{multicol}

\usepackage{hyperref}

\usepackage{booktabs}
\usepackage{graphicx}
\usepackage{tabularx}
\usepackage{caption,subcaption}

\usepackage{longtable}
\usepackage{tikz,tikz-cd}
\usetikzlibrary{matrix, calc, arrows, decorations.markings, decorations.pathmorphing, positioning}

\usepackage{amsmath,amssymb,amsthm}
\usepackage{stmaryrd,mathtools,mathdots}
\usepackage{eufrak}
\usepackage{mathabx}

\usepackage{accents}

\usepackage{bm,bbm}

\usepackage{braket}

\numberwithin{equation}{section}

\usepackage{subdepth}

\usepackage[style=alphabetic, backend=bibtex8, citestyle=alphabetic, firstinits=true, maxnames=99, maxalphanames=5, sorting=nyt]{biblatex}
\DeclareFieldFormat[article, inbook, incollection, inproceedings, patent, thesis, unpublished]{title}{#1} 
\renewbibmacro{in:}{}

\makeatletter
\renewcommand\section{\@startsection{section}{1}%
\z@{.7\linespacing\@plus\linespacing}{.5\linespacing}%
{\large\scshape\centering}}
\renewcommand\subsection{\@startsection{subsection}{2}%
  \z@{.5\linespacing\@plus.7\linespacing}{-.5em}%
  {\scshape}}
\makeatother

\makeatletter
\newcommand \Dotfill {\leavevmode \leaders \hb@xt@ 6pt{\hss .\hss }\hfill \kern \z@}
\makeatother
\makeatletter
\def\@tocline#1#2#3#4#5#6#7{\relax
  \ifnum #1>\c@tocdepth % then omit
  \else
    \par \addpenalty\@secpenalty\addvspace{#2}%
    \begingroup \hyphenpenalty\@M
    \@ifempty{#4}{%
      \@tempdima\csname r@tocindent\number#1\endcsname\relax
    }{%
      \@tempdima#4\relax
    }%
    \parindent\z@ \leftskip#3\relax \advance\leftskip\@tempdima\relax
    \rightskip\@pnumwidth plus4em \parfillskip-\@pnumwidth
    #5\leavevmode\hskip-\@tempdima
      \ifcase #1
       \or\or \hskip 1.65em \or \hskip 3.3em \else \hskip 4.95em \fi%
      #6\nobreak\relax
    \Dotfill
    \hbox to\@pnumwidth{\@tocpagenum{#7}}\par
    \nobreak
    \endgroup
  \fi}
\makeatother
\makeatletter
\def\l@section{\@tocline{1}{0pt}{1pc}{}{\scshape}}
\renewcommand{\tocsection}[3]{%
\indentlabel{\@ifnotempty{#2}{\ignorespaces#1 #2.\hskip 0.7em}}#3}
\def\l@subsection{\@tocline{2}{0pt}{1pc}{5pc}{\scshape}}

\def\l@subsubsection{\@tocline{3}{0pt}{1pc}{7pc}{\scshape}}

\makeatother

\newenvironment{myitemize}{%
\begin{list}{$\bullet$}%
 	{%
	\setlength{\itemsep}{0.4em}%
	\setlength{\topsep}{0.5em}%
	\setlength\leftmargin{2.45em}%
	\setlength\labelwidth{2.05em}%
	\setlength{\labelsep}{0.4em}%
	}%
	}%
{\end{list}}
\renewenvironment{itemize}{
\begin{myitemize}}%
{\end{myitemize}}

\MakePerPage{footnote}
\makeatletter
\newcommand*{\myfnsymbolsingle}[1]{%
  \ensuremath{%
    \ifcase#1% 0
    \or % 1
      \dagger
    \else % >= 2
      \@ctrerr  
    \fi
  }%   
}   
\makeatother

\usepackage{alphalph}
\newalphalph{\myfnsymbolmult}[mult]{\myfnsymbolsingle}{}

\theoremstyle{plain}
\newtheorem{theorem}{Theorem}[section]

\newtheorem{proposition}[theorem]{Proposition}
\newtheorem{corollary}[theorem]{Corollary}

\theoremstyle{definition}

\newtheorem{remark}[theorem]{Remark}

\makeatletter
\DeclarePairedDelimiter\abs{\lvert}{\rvert} % absolute value
\let\oldabs\abs
\def\abs{\@ifstar{\oldabs}{\oldabs*}}
\DeclarePairedDelimiterX{\norm}[1]{\lVert}{\rVert}{#1} % norm
\let\oldnorm\norm
\def\norm{\@ifstar{\oldnorm}{\oldnorm*}}
\DeclarePairedDelimiterX{\ceil}[1]{\lceil}{\rceil}{#1} % ceiling
\let\oldceil\ceil
\def\ceil{\@ifstar{\oldceil}{\oldceil*}}
\DeclarePairedDelimiterX{\floor}[1]{\lfloor}{\rfloor}{#1} % floor
\let\oldfloor\floor
\def\floor{\@ifstar{\oldfloor}{\oldfloor*}}
\makeatother

\renewcommand{\P}{\mathbb{P}} % probability
\newcommand{\E}{\mathbb{E}} % expectation
\newcommand{\1}{\mathbbm{1}} % indicator function
\newcommand{\diff}{\mathop{}\!\mathrm{d}}
\renewcommand{\i}{\mathrm{i}} % imaginary unit
\DeclareMathOperator{\e}{\mathrm{e}} % Euler number
\newcommand{\sfK}{\mathsf{K}}
\newcommand{\sfT}{\mathsf{T}}
\newcommand{\sfR}{\mathsf{R}}
\newcommand{\sfC}{\mathsf{C}}
\newcommand{\sfS}{\mathsf{S}}
\newcommand{\sfB}{\mathsf{B}}
\newcommand{\sfa}{\mathsf{a}}
\newcommand{\sfb}{\mathsf{b}}
\newcommand{\sfc}{\mathsf{c}}
\newcommand{\sfd}{\mathsf{d}}
\newcommand{\sfe}{\mathsf{e}}
\newcommand{\CsfK}{\widecheck{\mathsf{K}}}
\newcommand{\CsfS}{\widecheck{\mathsf{S}}}\newcommand{\CsfB}{\widecheck{\mathsf{B}}}\newcommand{\Csfa}{\widecheck{\mathsf{a}}}\newcommand{\Csfb}{\widecheck{\mathsf{b}}}\newcommand{\Csfc}{\widecheck{\mathsf{c}}}\newcommand{\Csfd}{\widecheck{\mathsf{d}}}
\newcommand{\Crho}{\widecheck{\rho}}
\newcommand{\Csigma}{\widecheck{\sigma}}
\newcommand{\Ctau}{\widecheck{\tau}}
\newcommand{\polymer}{\mathcal{P}}

\DeclareMathOperator{\id}{id}
\newcommand{\up}{\rm{up}}
\DeclareMathOperator{\invGamma}{invGamma}
\DeclareMathOperator{\type}{type}
\DeclareMathOperator{\GL}{GL}
\newcommand{\Zrepl}[1]{{%
  Z^{\mathsf{repl}}_{#1}%
}}

\renewcommand{\emptyset}{\varnothing}
\newcommand{\N}{\mathbb{N}} % natural numbers
\newcommand{\Z}{\mathbb{Z}} % integer numbers
\newcommand{\R}{\mathbb{R}} % real numbers
\newcommand{\C}{\mathbb{C}} % complex numbers
\renewcommand{\epsilon}{\varepsilon}
\renewcommand{\rho}{\varrho}
\renewcommand{\phi}{\varphi}

\DeclareMathSymbol{\widehatsym}{\mathord}{largesymbols}{"62}

\renewcommand{\tilde}{\widetilde}

\makeatletter
\newcommand{\doubletilde}[1]{{%
  \mathpalette\double@tilde{#1}%
}}
\newcommand{\double@tilde}[2]{%
  \sbox\z@{$\m@th#1\tilde{#2}$}%
  \ht\z@=.9\ht\z@
  \tilde{\box\z@}%
}
\makeatother

\newenvironment{myenumerate}{%
\renewcommand{\theenumi}{(\roman{enumi})}%
\renewcommand{\labelenumi}{\theenumi}%
\begin{list}{\labelenumi}
	{%
	\setlength{\itemsep}{0.4em}%
	\setlength{\topsep}{0.5em}%
	\setlength\leftmargin{2.45em}%
	\setlength\labelwidth{2.05em}%
	\setlength{\labelsep}{0.4em}%
	\usecounter{enumi}%
	}%
	}%
{\end{list}
}
{\end{myenumerate}}

  {\left\lbrace\begin{array}{@{}l@{}}}%
  {\end{array}\right.}

\makeatletter
\newsavebox{\mybox}\newsavebox{\mysim}
\newcommand{\distras}[1]{%
  \savebox{\mybox}{\hbox{\kern3pt$\scriptstyle#1$\kern3pt}}%
  \savebox{\mysim}{\hbox{$\sim$}}%
  \mathbin{\overset{#1}{\kern\z@\resizebox{\wd\mybox}{\ht\mysim}{$\sim$}}}%
}
\makeatother

\author[E.~Bisi]{Elia Bisi}
\address{Institut f\"ur Stochastik und Wirtschaftsmathematik\\
Technische Universit\"at Wien\\
E 105-07\\
Wiedner Hauptstrasse 8-10\\
1040 Wien\\
Austria}
\email{elia.bisi@tuwien.ac.at}

\author[N.~O'Connell]{Neil O'Connell}
\address{School of Mathematics and Statistics\\
University College Dublin\\
Dublin 4, Ireland}
\email{neil.oconnell@ucd.ie}

\author[N.~Zygouras]{Nikos Zygouras}
\address{Mathematics Institute\\
Zeeman building\\
University of Warwick\\
Coventry CV4 7AL, UK}
\email{n.zygouras@warwick.ac.uk}

\thanks{\emph{Acknowledgements.}
We thank Timo Sepp\"al\"ainen for useful discussions.
Elia Bisi and Neil O'Connell were supported by ERC grant 669306.
Nikos Zygouras was supported by EPSRC grant EP/R024456/1.}

\begin{document}

\title[The geometric Burge correspondence and the partition function of polymer replicas]{The geometric Burge correspondence \\ and the partition function of polymer replicas}

\begin{abstract}
We construct a geometric lifting of the Burge correspondence as a composition of local birational maps on generic Young-diagram-shaped arrays.
We establish its fundamental relation to the geometric Robinson-Schensted-Knuth correspondence and to the geometric Sch\"utzenberger involution.
We also show a number of properties of the geometric Burge correspondence, specializing them to the case of symmetric input arrays.
In particular, our construction shows that such a mapping is volume preserving in log-log variables.
As an application, we consider a model of two polymer paths of given length constrained to have the same endpoint, known as \emph{polymer replica}.
We prove that the distribution of the polymer replica partition function in a log-gamma random environment is a Whittaker measure, and deduce the corresponding Whittaker integral identity.
For a certain choice of the parameters, we notice a distributional identity between our model and the symmetric log-gamma polymer studied by O'Connell, Sepp\"al\"ainen, and Zygouras (2014).
\end{abstract}



\maketitle

%\vspace{-5mm}
\tableofcontents

\addtocounter{section}{0}

\newpage

\section{Introduction}
\label{sec:intro}

\subsection{Background}

The Robinson-Schensted-Knuth (RSK) correspondence~\cite{knuth70}, the Burge correspondence~\cite{burge74}, and the Sch\"utzenberger involution~\cite{schutzenberger77} are celebrated combinatorial bijections, classically described in terms of operations on (generalized) permutations, integer matrices, words, and Young tableaux.
These correspondences play a fundamental role in algebraic combinatorics, especially in the theory of symmetric functions.
See~\cite{fulton97, stanley99} for more details on these correspondences.

The RSK correspondence is classically described as a map between non-negative integer matrices and pairs of semistandard Young tableaux of the same shape, through a \emph{row insertion} algorithm.
It can be used to prove various Cauchy-Littlewood identities, thus connecting to Schur functions~\cite{stanley99}.
The RSK map and Schur functions underpin the solvability of various interconnected probabilistic models such as longest increasing subsequences in random permutations, directed last passage percolation, the corner growth model, and the totally asymmetric simple exclusion process -- see the seminal works~\cite{baikDeiftJohansson99, johansson00, baikRains01a}.

Fomin~\cite{fomin88} and Roby~\cite{roby91} first expressed the RSK correspondence in terms of local growth rules.
Berenstein and Kirillov~\cite{berensteinKirillov01} described it explicitly in terms of piecewise linear functions, i.e.\ operations in the $(\max,+)$-semiring, thus allowing the extension to input matrices with real (not necessarily integer) entries.
Furthermore, such a piecewise linear description is naturally prone to be extended to generic Young-diagram-shaped input arrays (not necessarily rectangular).
The latter aspect was useful to study last passage percolation models with point-to-line and point-to-half-line path geometries and/or various symmetries on the input weights.
In particular, Bisi and Zygouras~\cite{bisi18, bisiZygouras19b, bisiZygouras19c} found new exact formulas for such models in terms of all the irreducible characters of the classical groups (e.g.\ symplectic and orthogonal characters), which complemented the Schur measures of Baik and Rains~\cite{baikRains01a}.

The Burge correspondence is, in the classical combinatorial description, a variant of the RSK mapping that bijectively transforms a non-negative integer matrix into a pair of semistandard Young tableaux of the same shape, through a \emph{column insertion} algorithm.
Its description in terms of piecewise linear functions is due to van Leeuwen~\cite{vanLeeuwen05} -- see also~\cite{krattenthaler06, betea18}.
The Burge correspondence, analogously to RSK, can be used to study last passage percolation models; for recent applications, see~\cite{betea18, bisiCundenGibbonsRomik20}.

The Sch\"utzenberger involution is classically described as an \emph{evacuation} algorithm, or equivalently a sequence of \emph{jeu de taquin} operations on a (semi)standard Young tableau.
It can be shown to be indeed an involution and to preserve the shape of the tableau.
It can be alternatively described as a piecewise linear function on Gelfand-Tsetlin patterns~\cite{berensteinKirillov96} and, as such, extended to input patterns with real entries.
Sch\"utzenberger involution and \emph{jeu de taquin} have been also proven to be useful tools in combinatorial probability -- see e.g.~\cite{romikSniady15}.

Considering the piecewise linear description of the above bijections, one may formally replace the operations $(\max,+)$ of the ``tropical'' semiring with the operations $(+,\times)$ of the usual algebra.
Following~\cite{oConnellSeppalainenZygouras14}, we call such a procedure \emph{geometric lifting} and the resulting bijections \emph{geometric}, from the
theory of geometric crystals~\cite{berensteinKazhdan07}.
In particular, the \emph{geometric RSK correspondence} is a birational mapping on matrices with positive real entries, while the \emph{geometric Sch\"utzenberger involution} is a birational involution on triangular (more generally, trapezoidal) arrays with positive real entries.
They have been both introduced by Kirillov~\cite{kirillov01} and further studied by Noumi and Yamada~\cite{noumiYamada04}.

The geometric RSK correspondence and Whittaker functions, together, explain the solvability of certain $(1+1)$-dimensional models of \emph{random directed polymers in a random environment}.
These statistical physics models were introduced in~\cite{huseHenley85} and have been the object of intense research over the past thirty years -- see~\cite{comets17} for a recent review.
By \emph{directed lattice path} of length $n-1$ we mean any sequence $\bm{\pi}=(\pi(1),\dots,\pi(n))\in (\N^2)^n$ such that $\norm{\pi(i+1) - \pi(i)}_1 =1$ for $1\leq i\leq n-1$ and $\norm{\pi(n) - \pi(1)}_1 = n-1$; only two consecutive directions are thus allowed for the whole path, for example south and east (see Figure~\ref{fig:overlappingPaths}).
Denote by $\Pi_{m,n}$ the set of all directed lattice paths from $(1,1)$ to $(m,n)\in\N^2$, and let $(W_{i,j})_{(i,j)\in\N^2}$ be a field of positive independent random weights, known as random environment.
We define on $\Pi_{m,n}$ the quenched probability, known as \emph{point-to-point polymer measure},
\begin{align}
\label{eq:p2pPolymerMeasure}
\polymer_{m,n}(\bm{\pi})
:= \frac{1}{Z_{m,n}} \prod_{(i,j)\in \bm{\pi}} W_{i,j}
&&\text{for } \bm{\pi} \in \Pi_{m,n} \, ,
\end{align}
where the normalizing \emph{point-to-point polymer partition function} is
\begin{equation}
\label{p2pPolymerPartitionFn}
Z_{m,n} := \sum_{\bm{\pi} \in \Pi_{m,n}} \prod_{(i,j)\in \bm{\pi}} W_{i,j} \, .
\end{equation}
We remark that, in the statistical physics literature, the directed path $\pi$ is often viewed as the trajectory of a $(1+1)$-dimensional simple walk and the random variables $W_{i,j}$ are viewed as Gibbs weights, i.e.\ exponential functions of a random energy divided by the temperature of the system.

If one applies the geometric RSK map to the (random) input matrix $(W_{i,j})_{1\leq i\leq m, \, 1\leq j\leq n}$, it turns out that the output matrix contains the partition function~\eqref{p2pPolymerPartitionFn} as the entry $(m,n)$.
This is a general fact.
Furthermore, for some specific choices of the random environment, the properties of the geometric RSK map also permit finding exact formulas for the distribution of the partition function.
The most studied exactly solvable polymer model is known as the \emph{log-gamma polymer} and was introduced in~\cite{seppalainen12}.
In this model, all the weights follow an inverse-gamma law (so that the corresponding Gibbs weights are log-gamma distributed).

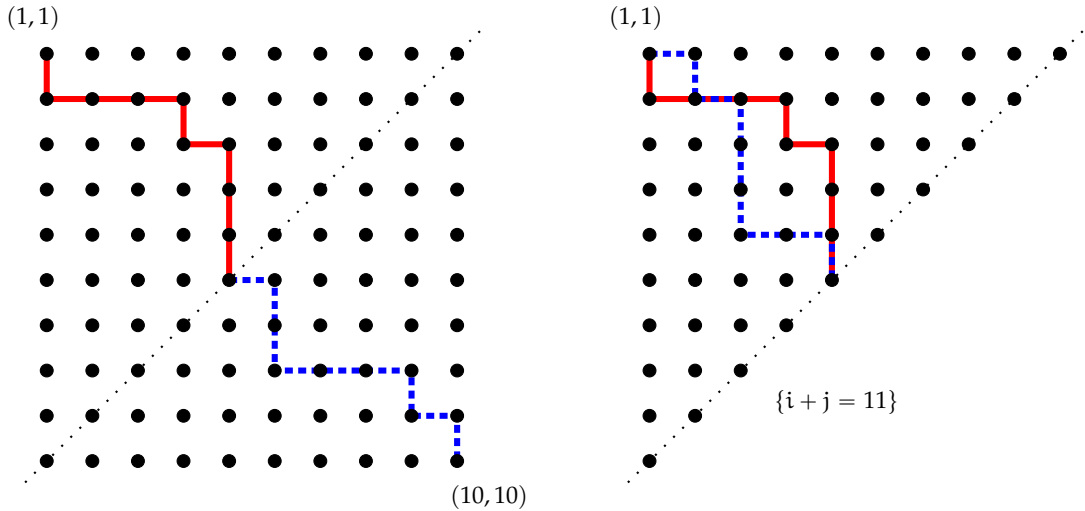
\begin{figure}
\centering
\begin{subfigure}[b]{.5\linewidth}
\centering
\captionsetup{width=.86\textwidth}%
\begin{tikzpicture}[scale=0.6, baseline, remember picture]
\draw[thick, loosely dotted] (10.5,-.5) -- (.5,-10.5);

\draw[line width=.8mm, color=red] (1,-1) -- (1,-2) -- (2,-2) -- (3,-2) -- (4,-2) -- (4,-3) -- (5,-3) -- (5,-4) -- (5,-5) -- (5,-6);

\draw[line width=.8mm, color=blue, densely dashed]
(5,-6) -- (6,-6) -- (6,-7) -- (6,-8) -- (7,-8) -- (8,-8) -- (9,-8) -- (9,-9) -- (10,-9) -- (10,-10);
	
\foreach \i in {1,...,10}{
\foreach \j in {1,...,10}{
		\node[draw,circle,inner sep=1.7pt,fill] at (\j,-\i) {};
	}
}

\node at (0.7,-0.2) {\footnotesize $(1,1)$};
\node at (10.7,-10.8) {\footnotesize $(10,10)$};
\end{tikzpicture}
\label{subfig:pointToPointPath}
\end{subfigure}%
\begin{subfigure}[b]{.5\linewidth}
\centering
\captionsetup{width=.86\textwidth}%
\begin{tikzpicture}[scale=0.6, baseline, remember picture]
\draw[thick, loosely dotted] (10.5,-.5) -- (.5,-10.5);

\draw[line width=.8mm, color=red] (1,-1) -- (1,-2) -- (2,-2) -- (3,-2) -- (4,-2) -- (4,-3) -- (5,-3) -- (5,-4) -- (5,-5) -- (5,-6);

\draw[line width=.8mm, color=blue, densely dashed] (1,-1) -- (2,-1) -- (2,-2) -- (3,-2) -- (3,-3) -- (3,-4) -- (3,-5) -- (4,-5) -- (5,-5) -- (5,-6);

\foreach \i in {1,...,10}{
		\foreach \j in {\i,...,10}{
		\node[draw,circle,inner sep=1.7pt,fill] at (11-\j,-\i) {};
	}
}

\node at (0.7,-0.2) {\footnotesize $(1,1)$};
\node at (5.1,-8.7) {\footnotesize $\{i+j=11\}$};
\vspace{10mm}
\end{tikzpicture}
\label{subfig:overlappingPaths}
\end{subfigure}%
\caption{On the left-hand figure, a directed lattice path from $(1,1)$ to $(n,n)$, with $n=10$.
The dotted antidiagonal line $\{i+j = n+1\}$ divides the path into two parts, denoted by a solid red line and a dashed blue line respectively.
Any such a path can be bijectively mapped into a pair of paths of length $n-1$ constrained to have the same endpoint.
This is illustrated on the right-hand figure, where the blue path is just reflected about the antidiagonal.}
\label{fig:overlappingPaths}
\end{figure}

Corwin, O'Connell, Sepp\"al\"ainen, and Zygouras~\cite{corwinOConnellSeppalainenZygouras14} linked the distribution of the log-gamma polymer partition function to Whittaker functions, in their integral formulation given by Givental~\cite{givental97}.
Their argument was based on the connection between the geometric RSK correspondence and $\GL_n(\R)$-Whittaker functions, analogous to the well-known relationship between the RSK map and Schur functions.
The analog of the Cauchy-Littlewood identity in this setting turned out to be a certain Whittaker integral identity due to Bump~\cite{bump89} and Stade~\cite{stade02}.

Subsequently, O'Connell, Sepp\"al\"ainen, and Zygouras~\cite{oConnellSeppalainenZygouras14} provided a new description of the geometric RSK as a composition of several local birational maps on the entries of the input matrix, deducing that the geometric RSK is volume preserving in log-log variables.
Besides, they began the study of the geometric RSK map in the presence of symmetry constraints, analyzing the corresponding polymer models and Whittaker measures.
This program drew inspiration from the work of Baik and Rains~\cite{baikRains01a} on RSK with symmetries, last passage percolation, and Schur measures, and aimed at studying their non-determinantal analogs in the polymer setting.
In particular, \cite{oConnellSeppalainenZygouras14} focused on symmetric input matrices, i.e.\ such that $W_{i,j}=W_{j,i}$ for all $i,j$, proving that the volume preserving property still holds in this setting.
This allowed studying the log-gamma polymer in a symmetric environment and obtaining the distribution of its partition function as a (different) Whittaker measure.
The corresponding Whittaker integral identity is equivalent to a formula for the Mellin transform of a $\GL_n(\R)$-Whittaker function due to Stade~\cite{stade01}.

The log-gamma polymer with point-to-line or point-to-half-line path geometries (with or without symmetry), where the polymer path has a fixed length but the endpoint is not fixed, has been analyzed by Bisi and Zygouras~\cite{bisiZygouras19a}.
Denote by $\Pi_{n}$ the set of all directed lattice paths of length $n-1$ starting at $(1,1)$.
Analogously to the point-to-point case, one can then define the \emph{point-to-line polymer measure}
\begin{align}
\label{eq:p2lPolymerMeasure}
\polymer_{n}(\bm{\pi})
:= \frac{1}{Z_{n}} \prod_{(i,j)\in \bm{\pi}} W_{i,j}
&&\text{for } \bm{\pi} \in \Pi_{n} \, ,
\end{align}
where the normalizing \emph{point-to-line polymer partition function} is
\begin{equation}
\label{p2lPolymerPartitionFn}
Z_{n} := \sum_{\bm{\pi} \in \Pi_{n}} \prod_{(i,j)\in \bm{\pi}} W_{i,j} \, .
\end{equation}
The main contribution of~\cite{bisiZygouras19a} was to express the law of $Z_n$ in terms of $\mathrm{SO}_{2n+1}(\R)$-Whittaker functions.
Their primary tool was the geometric RSK map extended to generic polygonal (not necessarily rectangular) input arrays, already used in~\cite{nguyenZygouras17}.

\subsection{Contributions of this work}
\label{subsec:intro_contribution}

In this work, we continue the program initiated in~\cite{oConnellSeppalainenZygouras14} of studying polymer models in symmetric environments and we focus on a \emph{persymmetric environment}, a case that was out of the scope of the approach of~\cite{oConnellSeppalainenZygouras14} and could not be covered therein.
Namely, we consider a weights' matrix $(W_{i,j})_{1\leq i,j\leq n}$ that is symmetric about the antidiagonal, i.e.\ such that $W_{i,j} = W_{n-j+1, n-i+1}$ for all $i,j$; a matrix with this property is usually called \emph{persymmetric}.
Notice that the \emph{point-to-point persymmetric polymer partition function} can be written as
\begin{equation}
\label{eq:overlappingPolymers}
Z_{n,n}
= \sum_{\bm{\pi} \in \Pi_{n,n}} \prod_{(i,j)\in \bm{\pi}} W_{i,j}
= \sum_{\substack{(a,b)\colon \\ a+b=n+1}}
\underbrace{\left(\sum_{\bm{\pi} \in \Pi_{a,b}}
\prod_{(i,j)\in \bm{\pi}} W'_{i,j} \right)}_{=Z'_{a,b}}
\underbrace{\left(\sum_{\tilde{\bm{\pi}} \in \Pi_{a,b}}
\prod_{(k,l)\in \tilde{\bm{\pi}}} W'_{k,l} \right)}_{=Z'_{a,b}} \, ,
\end{equation}
where $W'_{i,j}:= \sqrt{W_{i,j}}$ for $i+j=n+1$ and $W'_{i,j}:= W_{i,j}$ for $i+j<n+1$, so that $Z'_{a,b}$ is the point-to-point partition function with endpoint $(a,b)$ on the line $\{a+b=n+1\}$ and associated with the modified environment $(W_{i,j}')$.
The ``path transformation'' that justifies the identity above is illustrated in Figure~\ref{fig:overlappingPaths}.
Thus, remarkably, from the physical point of view, the point-to-point persymmetric polymer partition function can be interpreted as the \emph{replica partition function} for two polymer paths of length $n-1$ in the same environment $(W_{i,j}')$, constrained to coincide at the endpoint:
\begin{equation}
\label{eq:replica}
\Zrepl{n}
:= \sum_{\substack{(a,b)\colon \\ a+b=n+1}} \left(Z'_{a,b}\right)^2 \, .
\end{equation}
Replicas are important observables in 
statistical mechanics, as they can provide insights into the properties of the models.
For polymer models, replicas can shed light on localization phenomena~\cite{batesChatterjee20, comets17, chatterjee19, batesChatterjee20+, bates19}.
As a consequence of the connection with the persymmetric polymer, the present work leads to the computation of the Laplace transform of the replica partition function~\eqref{eq:replica}, which will be expressed as an integral of special functions called \emph{Whittaker functions} (more in-depth explanations will be given later on and the detailed formulas can be found in Section~\ref{sec:polymer}).

A key to the problem of studying the distribution of the persymmetric polymer partition function is to prove a \emph{volume preserving property} for the geometric RSK correspondence restricted to persymmetric matrices, as we now explain.
The image of a persymmetric matrix $\bm{w} = (w_{i,j})_{1\leq i,j\leq n}$ under the geometric RSK map is a matrix $\bm{t}= (t_{i,j})_{1\leq i,j\leq n}$ whose lower and upper triangular parts are Sch\"utzenberger dual of each other.
This is a consequence of the fact that the geometric RSK map commutes with the matrix transposition, together with Theorem~\ref{thm:gRSK-Schutz}.
Therefore, the map $(w_{i,j})_{i+j \leq n+1} \mapsto (t_{i,j})_{i\leq j}$ is a bijection.
Proving a volume preserving property for such a bijection would permit obtaining the distribution of the persymmetric polymer partition function as a Whittaker measure, using similar techniques as in~\cite{oConnellSeppalainenZygouras14}.
However, as the persymmetric constraints are ``non-local'' with respect to the order of composition of the local maps, it is not possible to prove the desired property from the geometric RSK construction as given in~\cite{oConnellSeppalainenZygouras14}.

Our alternative approach to the analysis of the persymmetric polymer, instead, will consist in constructing and studying what we call the \emph{geometric Burge correspondence}.
We define it as a sequence of local birational maps, as done in~\cite{oConnellSeppalainenZygouras14} for geometric RSK, via geometric lifting of the piecewise linear description of the combinatorial Burge correspondence presented (with minor differences) in~\cite{vanLeeuwen05, krattenthaler06, betea18}.
Notice that, wherever possible, we work with generic Young-diagram-shaped arrays instead of matrices.
One of our main contributions is Theorem~\ref{thm:gRSK-Burge-Schutz}, which links together the three geometric correspondences (RSK, Burge, and Sch\"utzenberger) via column/row mirror reflection of the input matrix.
Its combinatorial version, i.e.\ Theorem~\ref{thm:RSK-Burge-Schutz}, is a classical result.
However, the approach required to prove Theorem~\ref{thm:gRSK-Burge-Schutz} differs significantly from the known combinatorial proofs; in particular, we will use the description of the three geometric correspondences in terms of local maps and apply induction several times.
Interestingly, our proof reduces to certain local `commutation relations' -- see Proposition 3.3 -- that, to our knowledge, were not even known in the combinatorial setting.

Let us now connect this construction to polymer models.
Define $\Pi^{*}_{m,n}$ to be the set of all directed lattice paths from $(m,1)$ to $(1,n)$; notice that this set is ``dual'' to $\Pi_{m,n}$, in the sense that its paths connect the other pair of opposite vertices of the rectangle $[1,m]\times [1,n]$.
Define also the \emph{dual partition function} as
\begin{equation}
\label{p2pPolymerPartitionFn*}
Z^*_{m,n}
:= \sum_{\pi \in \Pi^{*}_{m,n}} \prod_{(i,j)\in \pi} W_{i,j} \, .
\end{equation}
We then have that the image $\bm{T} = (T_{i,j})_{1\leq i\leq m, \, 1\leq j\leq n}$ of the matrix $\bm{W} = (W_{i,j})_{1\leq i\leq m, \, 1\leq j\leq n}$ under the geometric Burge correspondence contains the dual partition function $Z^*_{m,n}$ as the entry $T_{m,n}$.
This will be an immediate consequence of Theorem~\ref{thm:gRSK-Burge-Schutz} and of the aforementioned fact that the geometric RSK output matrix contains the ``usual'' partition function $Z_{m,n}$ defined in~\eqref{p2pPolymerPartitionFn} as the entry $T_{m,n}$.

It is clear that the partition function on a persymmetric input matrix $\bm{W}$ coincides with the dual partition function on the symmetric input matrix obtained by reversing the rows of $\bm{W}$.
This, along with the observations above, explains why we can use the geometric Burge correspondence to study the persymmetric polymer partition function.
This approach turns out to be far more convenient, as we need to deal with symmetric (instead of persymmetric) input matrices.
As all the local maps of interest are volume preserving in log-log variables, the geometric Burge correspondence also is.
Furthermore, the geometric Burge correspondence, like the geometric RSK, behaves nicely when restricted to symmetric matrices: the image of a symmetric matrix is also symmetric, and the volume preserving property continues to hold almost trivially.
Using also other properties of the geometric Burge correspondence that we establish along the way (either via Theorem~\ref{thm:gRSK-Burge-Schutz} or via the local maps definition), we will be able to obtain the distribution of the persymmetric polymer partition function as a $\GL_n(\R)$-Whittaker measure.

The Whittaker measure that we find for the persymmetric polymer coincides, for a certain choice of parameters, with the one for the symmetric polymer obtained in~\cite{oConnellSeppalainenZygouras14}.
This seems to be a highly non-trivial fact: we are not aware of a direct proof based on the definition of the polymer models.
We also mention that a number of other very interesting distributional identities in integrable polymer models have been observed in recent papers: \cite{borodinGorinWheeler19, galashin20} are based on six vertex models and Yang-Baxter equations (see also~\cite{borodinBufetov20+} for related work), whereas~\cite{dauvergne20} relies on RSK methods.
However, the distributional equality we have observed in this work, for polymers on \emph{symmetric} input matrices, does not appear in the above works.

As we have described above, our construction of the geometric Burge correspondence allows us to connect to polymer models.
We should mention that the \emph{combinatorial} Burge correspondence has been already used to deal with last passage percolation models, which (in the statistical physics terminology) are `zero temperature degenerations' of polymer partition functions.
This is e.g.\ the case in~\cite{betea18}, where the main focus is on last passage percolation with point-to-half-line path geometry.
Actually, one could argue that the Burge correspondence had also been implicitly used in~\cite{baikRains01a} to provide a combinatorial bijection for the RSK map restricted to persymmetric matrices and ultimately study persymmetric last passage percolation.
Our approach can, thus, be considered as a geometric lifting of Baik and Rains's construction.
However, in the combinatorial setting everything can be phrased in terms of the Sch\"utzenberger involution and there is no reason to give too much attention to the Burge map itself.
On the other hand, in the geometric setting, in order to obtain the required volume preserving property for the geometric RSK map restricted to persymmetric matrices, one requires a much deeper understanding of the geometric Burge map itself and its relation to the geometric versions of the RSK and Sch\"utzenberger maps.
This is contained in our construction via local maps together with Proposition~\ref{prop:compDiagMaps} and Theorem~\ref{thm:gRSK-Burge-Schutz}.
Proposition~\ref{prop:compDiagMaps}, which is one of the key ingredients used in the proof of Theorem~\ref{thm:gRSK-Burge-Schutz}, gives a remarkable (and seemingly non-trivial) relation satisfied by the local maps involved.

In terms of asymptotic analysis, the partition functions of several polymer models are expected to be in the \emph{KPZ universality class}.
Although there has recently been important progress in this regard~\cite{dimitrov20, virag20}, this doesn't cover the setting of polymers with symmetries.
For the latter, formal asymptotics have been achieved in~\cite{barraquandBorodinCorwin20}.
We do not address such issues in the present work.

\vskip 2mm

\noindent
{\bfseries Organization of the article.}
In Section~\ref{sec:combinatorialCorrespondences} we introduce some notation and recap known piecewise linear descriptions of the classical combinatorial RSK, Burge, and Sch\"utzenberger bijections, to prepare the reader for the geometric lifting.
In Section~\ref{sec:geometricCorrespondences} we introduce the geometric Burge correspondence as a composition of local birational maps, recall analogous (known) definitions of the geometric RSK and Sch\"utzenberger maps, and explain their interconnection.
In Section~\ref{sec:Burge_properties} we prove that the geometric Burge correspondence is volume preserving in log-log variables, as well as other useful properties.
Section~\ref{sec:Burge_symmetric} deals with the restriction of the geometric Burge correspondence to symmetric input arrays and the specialization of its properties in this setting.
In Section~\ref{sec:polymer} we consider the persymmetric polymer (or equivalently the replica) partition function, proving that its distribution is given by a $\GL_n(\R)$-Whittaker measure and deducing the corresponding Whittaker integral identity; we also discuss
the relation to the symmetric polymer studied in~\cite{oConnellSeppalainenZygouras14}.

\section{Combinatorial bijections and their piecewise linear formulation}
\label{sec:combinatorialCorrespondences}

The RSK, Burge and Sch\"utzenberger correspondences are combinatorial bijections, classically constructed via row insertion, column insertion, and \emph{jeu de taquin} operations, respectively.
In this section, we give a brief and expository reminder of these bijections and their reformulation in terms of piecewise linear transformations. This will motivate their geometric lifting that we perform in Section~\ref{sec:geometricCorrespondences}. Besides the classical references, e.g.~\cite{fulton97}, we refer to~\cite{voWhitney83} for several combinatorial aspects and interesting relations between these correspondences.
We also refer to~\cite{berensteinKirillov96, noumiYamada04, krattenthaler06, oConnellSeppalainenZygouras14, zygouras18} for more details on the RSK correspondence and its piecewise linear formulations.
Piecewise linear descriptions of the Burge correspondence (often exposed in the formalism of Fomin growth diagrams) can be found in~\cite{vanLeeuwen05, krattenthaler06, betea18}; in particular, the version that we present here is closer to~\cite{betea18}.
Finally, we refer to~\cite{berensteinKirillov96} for the piecewise linear formulation of the Sch\"utzenberger involution.

\vskip 2mm

Let $\N:=\{1,2,3,\dots\}$.
We will view any \textbf{Young diagram} $\lambda$ as the partition $(\lambda_1,\lambda_2,\dots)$ such that $\lambda_i$ is the number of boxes in the $i$-th row of $\lambda$ or, equivalently, as the index set $\{(i,j)\in\N^2 \colon j\leq \lambda_i\}$ of its boxes.
We will say that $(m,n)$ is a \textbf{border box} of a Young diagram $\lambda$ if it is the last box of the corresponding diagonal, i.e.\ if $(m+1,n+1) \notin\lambda$.
In particular, we will call $(m,n)$ a \textbf{corner box} if $\lambda\setminus\{(m,n)\}$ is a Young diagram, i.e.\ if none of the three boxes $(m,n+1)$, $(m+1,n)$, $(m+1,n+1)$ belongs to $\lambda$.
For example, for the partition $ (2,2,1) \equiv \{(1,1), (1,2), (2,1), (2,2), (3,1)\}$, all boxes except $(1,1)$ are border boxes, but only $(2,2)$ and $(3,1)$ are corner boxes.
We will also denote a {\bfseries rectangular Young diagram} by $m\times n := \{1,\dots,m\} \times \{1,\dots,n\}$.

For a given Young diagram $\lambda$, let $\R^{\lambda}$ be the set of $\lambda$-shaped arrays $P = (p_{i,j})_{(i,j)\in\lambda}$ of real numbers.
If the values $p_{i,j}$ are restricted to be positive integers and also have the property that are weakly increasing in $j$, for any fixed $i$, and strictly increasing in $i$, for any fixed $j$, then $P$ is called a {\bf semistandard Young tableau}.
A useful reparametrization of Young tableaux goes under the name of \textbf{Gelfand-Tsetlin patterns}.
Given a Young tableau $P=(p_{i,j})_{(i,j)\in\lambda}$, its corresponding Gelfand-Tsetlin pattern $\bm{u} = ( u_{i,j})_{i,j\geq 1}$ is given by
\begin{equation}
\label{eq:GT}
u_{i,j} := \#\{1\leq k\leq \lambda_j \colon p_{j,k}\leq i\} \, .
\end{equation}
In words, $u_{i,j}$ is the number of entries in the $j$-th row of $P$ that are less than or equal to $i$.
Assuming that the shape $\lambda$ of $P$ is of length at most $m$ and the entries $p_{i,j}$ are in the alphabet $\{1,\dots,n\}$, one can view $\bm{u}$ as a trapezoidal array $(u_{i,j})_{1\leq i\leq n, \, 1\leq j\leq i \wedge m}$ -- the $u_{i,j}$'s not in this range of indices being redundant.
By construction, the shape $\lambda$ of $P$ corresponds to the bottom row $(u_{n,1},u_{n,2},\dots)$ of $\bm{u}$.
Moreover, it is easy to verify that the Gelfand-Tsetlin variables, as defined in~\eqref{eq:GT}, satisfy the interlacing conditions
\[
u_{i+1,j+1}\leq u_{i,j}\leq u_{i+1,j} \, .
\]

\vskip 2mm
\noindent
{\bfseries Piecewise linear maps.}
We collect here, for convenience, all the piecewise linear maps that represent the building blocks of the piecewise linear formulation of all our combinatorial maps (RSK, Burge, and Sch\"utzenberger).
Such a formulation, though, will be introduced for each combinatorial map separately, later in this section.
We will denote all the piecewise linear maps by letters with a ``vee accent'', to distinguish them from the corresponding maps in the geometric setting of the next sections.

For a Young diagram $\lambda$ and $(i,j)\in\lambda$, we define $\Csfa_{i,j}, \Csfb_{i,j}, \Csfc_{i,j} \colon \R^{\lambda} \to \R^{\lambda}$ as the \textbf{local maps} that act on $\bm{w}\in \R^{\lambda}$ by only modifying $w_{i,j}$ according to the following rules:
\begin{align}
\label{eq:a_comb}
&\Csfa_{i,j}\colon w_{i,j} \longmapsto w_{i-1,j} \vee w_{i,j-1} + w_{i+1,j} \wedge w_{i,j+1} - w_{i,j} \, , \\
\label{eq:b_comb}
&\Csfb_{i,j}\colon w_{i,j} \longmapsto w_{i-1,j} \vee w_{i,j-1} + w_{i,j+1} - w_{i,j} \, , \\
\label{eq:c_comb}
&\Csfc_{i,j}\colon w_{i,j} \longmapsto w_{i-1,j} \vee w_{i,j-1} + w_{i,j}  \, .
\end{align}
For two distinct indices $(i,j),(k,l)\in\lambda$, we also define $\Csfd^{k,l}_{i,j} \colon \R^{\lambda} \to \R^{\lambda}$ as the local map that acts on an array $\bm{w}\in \R^{\lambda}$ by only modifying $w_{i,j}$ and $w_{k,l}$ as follows:
\begin{equation}
\label{eq:d_comb}
\Csfd^{k,l}_{i,j}\colon
\begin{cases}
w_{i,j} \longmapsto (w_{i-1,j} \vee w_{i,j-1} + w_{k,l}) \wedge w_{i,j} \, , \\
w_{k,l} \longmapsto w_{k,l} - w_{k,l} \wedge (w_{i,j} - w_{i-1,j} \vee w_{i,j-1}) + w_{i+1,j} \wedge w_{i,j+1} - w_{i,j} \, .
\end{cases}
\end{equation}
For $i=1$ and/or $j=1$, the values of $w_{i-1,j}$ and $w_{i,j-1}$ are determined by the following convention: $w_{0,1}=w_{1,0}=0$ and $w_{0,k}=w_{k,0}=-\infty$ for all $k>1$.
Notice that $\Csfa_{i,j}$ and $\Csfd_{i,j}^{k,l}$ also involve entries $w_{i+1,j}$ and $w_{i,j+1}$, so for these maps to be well defined we assume that $(i+1,j),(i,j+1) \in \lambda$; likewise, for $\Csfb_{i,j}$ we assume that $(i,j+1) \in \lambda$.
It is clear that all the above local maps are bijective, but only $\Csfa_{i,j}$ and $\Csfb_{i,j}$ are involutions.
Furthermore, they all satisfy several trivial commutative properties, as each modifies only one or two entries of the input array.

From now on, for any $n\in\Z$ we will refer to the $n$-th diagonal of an array $\bm{w}\in \R^{\lambda}$ as the sequence of its entries $w_{i,j}$ such that $j-i=n$.
Let us now define, for all $(k,l)\in\lambda$, the following \textbf{diagonal maps}:
\begin{align}
\label{eq:rho_comb}
\Crho_{k,l} &:= \Csfa_{k-h+1,l-h+1} \circ \Csfa_{k-h+2,l-h+2} \circ \dots \circ \Csfa_{k-1,l-1} \circ \Csfc_{k,l} \, , \\
\label{eq:sigma_comb}
\Csigma_{k,l} &:= \Csfa_{k-h+1,l-h+1} \circ \Csfa_{k-h+2,l-h+2} \circ \dots \circ \Csfa_{k-1,l-1} \circ \Csfb_{k,l} \, , \\
\label{eq:tau_comb}
\Ctau_{k,l} &:= \Csfc_{k,l} \circ \Csfd^{k,l}_{k-1,l-1} \circ \dots \circ \Csfd^{k,l}_{k-h+2,l-h+2} \circ \Csfd^{k,l}_{k-h+1,l-h+1} \, ,
\end{align}
where $h:=k \wedge l$.
The terminology ``diagonal map'' comes from the fact that any of these maps indexed by $(k,l)$ only modifies the $(l-k)$-th diagonal of the input array.
It is likewise clear that any two diagonal maps commute if they act on diagonals that are not the same nor neighboring.
As compositions of bijections, diagonal maps are all bijective.
Furthermore, notice that $\Csigma_{k,l}$ is a composition of commuting involutions, hence it is itself an involution.
Diagonal maps~\eqref{eq:rho_comb}, \eqref{eq:sigma_comb}, and~\eqref{eq:tau_comb} will be involved in the construction of the RSK,  Burge, and Sch\"utzenbeger correspondences, respectively.

\vskip 2mm
\noindent
{\bfseries The Robinson-Schensted-Knuth (RSK) correspondence.}
This correspondence is based on an algorithm called {\bfseries row insertion}, which we will now describe.
Row inserting a positive integer $i$ into a given semistandard Young tableau works as follows: if $i$ is larger than or equal to all the entries of the first row of the tableau, then a new box containing $i$ is added at the end of the first row and the procedure stops.
Otherwise, $i$ replaces the \emph{first} number of the first row that is strictly larger than $i$.
The replaced number, call it $j$, is now ``bumped'' and inserted into the second row of the tableau in the same way.
The procedure continues until one of the bumped numbers is placed at the end of a row of the tableau, yielding a new semistandard Young tableau with one extra box.

Now, any \emph{word} $w$ in the alphabet $\{1,\dots,n\}$ can be decomposed into a sequence of $m$ increasing words $w_1$, \dots, $w_m$ (for some $m$):
\begin{equation}
\label{eq:word}
w = \underbrace{1^{w_{1,1}} 2^{w_{1,2}} \cdots n^{w_{1,n}}}_{w_1} \underbrace{1^{w_{2,1}} 2^{w_{2,2}} \cdots n^{w_{2,n}}}_{w_2} \cdots \underbrace{1^{w_{m,1}} 2^{w_{m,2}} \cdots n^{w_{m,n}}}_{w_m} \, ,
\end{equation}
where all $w_{i,j}$'s are non-negative integers and $i^r\equiv i\cdots i$ denotes a sequence of $r$ consecutive letters $i$.
The RSK algorithm acts on a word by successively row inserting all its letters.
More precisely, one starts by inserting the first letter of $w$ into the empty tableau $P_0=\emptyset$, thus obtaining a tableau $P_1$; then one inserts the second letter of $w$ into $P_1$, obtaining a new tableau $P_2$.
The process continues in the same way until all letters of $w$ have been inserted, thus yielding a tableau $P$ with as many boxes as the length of $w$.
In parallel with the $P$-tableau, one can construct a $Q$-tableau, which records the shapes of the successive sequence of (intermediate) $P$-tableaux after the row insertion of each increasing word $w_k$.
Namely, every time a letter of $w_k$ in inserted into the $P$-tableau, thus yielding a new $P$-tableau with one extra box, a box containing $k$ is also added to the $Q$-tableau in the same position.
The RSK algorithm thus yields a bijection between a word $w$ in the alphabet $\{1,\dots,n\}$ with $m$ increasing subwords and a pair of semistandard Young tableaux $(P,Q)$ of the same shape $\lambda$ of length at most $m\wedge n$, such that the entries of $P$ are in $\{1,\dots,n\}$ and the entries of $Q$ are in $\{1,\dots,m\}$.

Notice that the word $w$ is naturally encoded by a matrix $\bm{w}=(w_{i,j})_{1\leq i\leq m, \, 1\leq j \leq n} \in \Z_{\geq 0}^{m\times n}$.
Let now $\bm{u}$ and $\bm{v}$ be the Gelfand-Tsetlin patterns that bijectively correspond to the output tableaux $P$ and $Q$, respectively:
\begin{equation}
\label{eq:uv}
\bm{u} = ( u_{i,j})_{1\leq i\leq n, \, 1\leq j\leq i \wedge m} \, ,
\qquad\qquad
\bm{v} = ( v_{i,j})_{1\leq i\leq m, \, 1\leq j\leq i\wedge n } \, .
\end{equation}
As $P$ and $Q$ are of the same shape, we can ``glue'' together $\bm{u}$ and $\bm{v}$ along their identical bottom rows $(u_{n,1}, u_{n,2}, \dots) = (v_{m,1}, v_{m,2}, \dots) = \lambda$.
The result of this ``gluing'' is an $m\times n$ matrix, which we denote\footnote{This notation should not be confused with the notation for a skew partition $\lambda/\mu$.} by $\bm{t} = (\bm{u} \backslash \bm{v})$, defined through the following relations:
\begin{equation}
\label{eq:uvt}
u_{i,j} := t_{m-j+1,i-j+1} \, ,
\qquad\qquad
v_{i,j} := t_{i-j+1,n-j+1} \, .
\end{equation}
We will also say that $\bm{u}$ and $\bm{v}$ are the {\bfseries lower and upper triangular/trapezoidal parts} of $\bm{t}$, respectively.
For instance, when $m<n$, we have that
\begin{equation}
\label{eq:RSKoutputMatrix}
\bm{t} = (\bm{u} \backslash \bm{v})
=
\begin{tikzpicture}[baseline=(current bounding box.center)]
\matrix (m) [matrix of math nodes,nodes in empty cells,right delimiter={)},left delimiter={(} ]{
u_{m,m} && &u_{n,m}=v_{m,m} & &v_{1,1} \\
& & & & & \\
& & & & & \\
& & & & & \\
u_{1,1} && &u_{m,1} & &u_{n,1}=v_{m,1} \\
} ;
\draw[loosely dotted] (m-1-1)-- (m-5-1);
\draw[loosely dotted] (m-1-1)-- (m-5-4);
\draw[loosely dotted] (m-5-1)-- (m-5-4);
\draw[loosely dotted] (m-1-1)-- (m-1-4);
\draw[loosely dotted] (m-1-4)-- (m-5-6);
\draw[loosely dotted] (m-5-4)-- (m-5-6);
\draw[loosely dotted] (m-1-6)-- (m-5-6);
\draw[loosely dotted] (m-1-4)-- (m-1-6);
\end{tikzpicture}
\, .
\end{equation}
In this sense, the RSK correspondence can be also viewed as a map $\bm{w} \mapsto \bm{t}$ between $m\times n$ non-negative integer matrices.
Notice that the pattern $\bm{v} := ( v_{i,j})_{1\leq i\leq m, \, 1\leq j\leq i\wedge n }$ corresponding to $Q$ has the property that the diagonal $(v_{k,1},\dots,v_{k,k\wedge n })$ is the shape of the tableau obtained after the row insertion of $w_k$.
 
Recall that entry $u_{i,j}$ of the Gelfand-Tsetlin pattern $\bm{u}$, which corresponds to tableau $P$, encodes the numbers of entries in the $j$-th row of $P$ that are smaller than or equal to $i$.
Using this fact, we will now briefly describe how the combinatorially defined RSK correspondence translates to piecewise linear transformations on Gelfand-Tsetlin patterns.
Suppose that, after inserting the first $k-1$ words, we have obtained an (intermediate) $P$-tableau corresponding to the Gelfand-Tsetlin pattern $\bm{u}=(u_{i,j})_{1\leq i\leq n, \, 1\leq j\leq i \wedge (k-1)}$.
Next, row inserting word $w_k$ has the following effect: the number of ones in the $P$-tableau will increase to $u_{1,1}'= u_{1,1}+w_{k,1}$, after the insertion of $1^{w_{k,1}}$.
The inserted ones will bump a number of twos from the first row, which will then be row inserted in the second row.
The number of twos that are bumped equals $(u_{11}'-u_{1,1}) \wedge ( u_{2,1}-u_{1,1})= u_{1,1}' \wedge u_{2,1}-u_{1,1}$; as a result, the number of twos in the second row will become $u_{22}' = u_{22} + (u_{11}'\wedge u_{2,1} - u_{11})$. 
After that, the row insertion of $2^{w_{k,2}}$ will change $u_{2,1}$ to $u_{2,1}'=w_2 + u_{2,1}\vee u_{1,1}'$.
The row insertion of the twos leads to a bumping of threes and the process continues in the same fashion through analogous piecewise linear transformation.
Consider now transformations $\Crho_{k,l}$ defined in~\eqref{eq:rho_comb}.
The change of $u_{1,1}$ to $u_{1,1}'$ is encoded through the application of $\Crho_{k,1}$, while the change of $u_{2,1}$ and $u_{2,2}$ to  $u_{2,1}'$ and $u_{2,2}'$, respectively, is encoded through the application of $\Crho_{k,2}$.
Similarly, transformation $\Crho_{k,3}$ will encode the changes of $u_{3,1}$, $u_{3,2}$, and $u_{3,3}$ after the insertion of the threes and the corresponding bumping process, and so on.

As the expository description above suggests, the Gelfand-Tsetlin pattern $\bm{u}$ can be constructed by repeatedly applying maps of type $\Crho_{k,l}$.
Actually, the same happens for the pattern $\bm{v}$ that corresponds to the $Q$-tableau, as we now briefly argue.
A remarkable \emph{symmetry property} of the RSK algorithm is that transposing the input matrix $\bm{w}$ amounts to swapping the roles of the resulting Gelfand-Tsetlin patterns $\bm{u}$ and $\bm{v}$ (or equivalently the $P$- and $Q$-tableaux).
Therefore, $\bm{v}$ can be constructed via the same diagonal maps $\Crho_{k,l}$'s applied on the transposed matrix.
We conclude that the whole output matrix $\bm{t} = (\bm{u} \backslash \bm{v})$ is obtained via repeated applications of these diagonal maps.

At this point, RSK can be extended in two natural ways: firstly, it can be seen as a bijective map between arrays indexed by a Young diagram (the case of matrices thus corresponding to rectangular Young diagrams); secondly, it can be seen as a map between arrays with \emph{real} entries instead of non-negative integer entries (as the piecewise linear maps are still well defined).
In this general framework, the RSK correspondence $\CsfK^{\lambda}\colon\R^{\lambda} \to \R^{\lambda}$ on real $\lambda$-shaped arrays\footnote{The choice of the letter $\CsfK$ refers to Knuth.} is defined by
\begin{align}
\label{eq:RSK}
\CsfK^{\lambda} &:= \Crho_{i_n,j_n} \circ \Crho_{i_{n-1},j_{n-1}} \circ \dots \circ \Crho_{i_1,j_1} \, ,
\end{align}
where $((i_1,j_1),\dots,(i_n,j_n))$ is \emph{any} sequence of distinct boxes of $\lambda$ such that, for all $1\leq k\leq n$, $\lambda^{(k)} := \{(i_1,j_1),\dots, (i_k,j_k)\}$ is a Young diagram, and $\lambda^{(n)} = \lambda$.
Notice that for such a sequence there might be several choices, which all lead to equivalent definitions of $\CsfK^{\lambda}$, due to the commutative properties of the diagonal maps.
For instance, we have that
\[
\CsfK^{(2,1,1)} = \Crho_{3,1} \circ \Crho_{2,1} \circ \Crho_{1,2} \circ \Crho_{1,1}
= \Crho_{3,1} \circ \Crho_{1,2} \circ \Crho_{2,1} \circ \Crho_{1,1}
= \Crho_{1,2} \circ \Crho_{3,1} \circ \Crho_{2,1} \circ \Crho_{1,1} \, .
\]
Even though $\CsfK^{\lambda}$ is defined for $\lambda$-shaped arrays, it may also be applied to $\mu$-shaped arrays, for any $\mu \supseteq \lambda$, by acting on the $\lambda$-part of the diagram and leaving all entries indexed by $\mu/ \lambda$ unchanged; in such a case, we do \emph{not} use the simplified notation $\CsfK$, to avoid ambiguity.

Finally, let us highlight that the RSK correspondence features a simple recursive definition in terms of corner boxes: if $(m,n)$ is a corner box of $\lambda$, then
\[
\CsfK^{\lambda} =
\Crho_{m,n} \circ \CsfK^{\lambda\setminus\{(m,n)\}} \, ,
\]
where by convention $\sfK^{\emptyset}=\id$.

\vskip 2mm
\noindent
{\bfseries The Burge correspondence.}
This correspondence is based on an algorithm called {\bf column insertion}, somehow ``dual'' to row insertion.
To column insert a positive integer $i$ into a given semistandard Young tableau, one proceeds as follows.
If $i$ is \emph{strictly} larger than all the entries of the first \emph{column}\footnote{As opposed to the first \emph{row} in the case of row insertion.} of the tableau, then a new box containing $i$ is added at the end of the first column and the procedure stops.
Otherwise, $i$ replaces the first number of the first column that is larger than or equal to $i$.
The replaced number, call it $j$, is now ``bumped'' and inserted into the second column of the tableau in the same way.
The procedure continues until one of the bumped numbers is placed at the end of a column of the tableau, yielding a new semistandard Young tableau with one extra box (it is fairly easy to check that the insertion preserves the strict monotonicity of columns and the weak monotonicity of rows).

The Burge correspondence uses the column insertion (instead of the row insertion in the RSK correspondence) to map a word $w$, as in~\eqref{eq:word}, onto a pair of Young tableaux $(P,Q)$ of the same shape.
A \emph{second difference} from the RSK correspondence is that the letters of each increasing word $w_k$ are read in \emph{reverse order}; namely, the Burge correspondence successively column inserts the letters of $w$, reading them as follows:
\begin{equation}
\label{eq:wordRev}
\underbrace{n^{w_{1,n}} \cdots 2^{w_{1,2}} 1^{w_{1,1}}}_{w_1 \text{ reverse}}
\underbrace{n^{w_{2,n}}\cdots 2^{w_{2,2}} 1^{w_{2,1}}}_{w_2 \text{ reverse}}
\cdots
\underbrace{n^{w_{m,n}}\cdots 2^{w_{m,2}} 1^{w_{m,1}}}_{w_m \text{ reverse}} \, .
\end{equation}
To obtain the $P$-tableau, one constructs a sequence $(\emptyset=P_0, P_1, P_2 \dots, P_L = P)$ of $L$ intermediate tableaux (where $L=\sum_{i,j} w_{i,j}$ is the length of the letter $w$) that starts from the empty tableau and ends at the final $P$-tableau: for $1\leq l\leq L$, $P_l$ is obtained from $P_{l-1}$ by column inserting the $l$-th letter of~\eqref{eq:wordRev} into $P_{l-1}$.
Finally, similarly to RSK, the $Q$-tableau in the Burge correspondence records the sequence of shapes of the intermediate $P$-tableaux after the column insertion of each increasing word $w_k$ (read in reverse order).

Similarly to RSK, the Burge correspondence can be equivalently viewed as a bijection between a matrix $\bm{w}=(w_{i,j})_{1\leq i \leq m,\, 1\leq j\leq n}\in\Z_{\geq 0}^{m\times n}$ and a pair $(\bm{u},\bm{v})$ of Gelfand-Tsetlin patterns~\eqref{eq:uv} with the same bottom row.
Again, the two patterns can be glued together into an matrix $\bm{t} = (\bm{u} \backslash \bm{v})$ as in~\eqref{eq:uvt}-\eqref{eq:RSKoutputMatrix}, in order to view the Burge correspondence as a map between $m\times n$ non-negative integer matrices.

Let us now briefly describe how the combinatorial description of the Burge correspondence can be viewed as a sequence of piecewise linear transformations on Gelfand-Tsetlin patterns.
Suppose that we have inserted the first $k-1$ words, thus obtaining an intermediate $P$-tableau corresponding to a Gelfand-Tsetlin pattern $\bm{u}=(u_{i,j})_{1\leq i\leq n, \, 1\leq j\leq i\wedge (k-1)}$.
We now column insert word $w_k$, reading it in reverse order, i.e.\ as $n^{w_{k,n}}\cdots 2^{w_{k,2}}1^{w_{k,1}}$.
To convey the main idea in the simplest case, let us assume that $n=2$, so that the intermediate $P$-tableau has at most two rows, filled with ones and twos only.
When we column insert the first few twos, we start filling with the twos the $u_{2,2} - u_{1,1}$ ``free spaces'' in the second row of $P$.
However, we cannot insert more than $u_{2,2} - u_{1,1}$ twos in the second row of $P$, otherwise the strict monotonicity of the columns would be violated; hence, if at some point we run out of such ``free spaces'', the extra $w_{k,2} - w_{k,2} \wedge (u_{1,1} - u_{2,2})$ twos will end up in the first row.
As a result, on the one hand the number $u_{2,2}$ of twos in the second row changes to $u_{2,2}'= (w_{k,2} + u_{2,2}) \wedge u_{1,1}$; on the other hand, the total length $u_{2,1}$ of the first row of $P$ increases to $u'_{2,1} = u_{2,1} + (w_{k,2} - w_{k,2} \wedge (u_{1,1} - u_{2,2}))$.
After the twos have been inserted, we column insert the $w_{k,1}$ ones, placing them all in the first row.
This has the effect of increasing by $w_{k,1}$ both the number of ones of the first row and the total length of the first row.
Namely, $u_{1,1}$ changes to $u_{1,1}' = u_{1,1} + w_{k,1}$ and $u'_{2,1}$ changes to the final value $u_{2,1}'' = u'_{2,1}+ w_{k,1} = u_{2,1} + w_{k,2} - w_{k,2} \wedge (u_{1,1} - u_{2,2}) + w_{k,1}$.
The change of $u_{1,1}$ to $u_{1,1}'$ represents the action of transformation $\Ctau_{k,1} := \Csfc_{k,1}$ from~\eqref{eq:c_comb} and~\eqref{eq:tau_comb}, whereas the change of $(u_{2,1}, u_{2,2})$ to $(u_{2,1}'', u_{2,2}')$ can be viewed as the action of transformation $\Ctau_{k,2} := \Csfc_{k,2}\circ \Csfd_{k-1,1}^{k,2}$ from~~\eqref{eq:c_comb}, \eqref{eq:d_comb}, and~\eqref{eq:tau_comb} (the index $k$ refers to the fact that the $k$-th increasing word $w_k$ is being inserted).

The construction of $\bm{u}$ can thus be achieved via repeated applications of $\Ctau_{k,l}$'s.
As for RSK, a \emph{symmetry property} holds for the Burge correspondence as well: if $\bm{w}$ is mapped onto $(\bm{u},\bm{v})$, then the transpose of $\bm{w}$ is mapped onto $(\bm{v},\bm{u})$.
Therefore, $\bm{v}$ can be constructed by applying maps $\Ctau_{k,l}$'s on the transposed matrix.
We conclude that the whole output matrix $\bm{t} = (\bm{u} \backslash \bm{v})$ is obtained via repeated applications of these diagonal maps.

Thanks to its piecewise linear description, the Burge correspondence can be also extended as a bijection between Young-diagram-shaped arrays with real entries.
In this general framework, it is defined\footnote{The choice of the letter $\CsfB$ refers to Burge.} as the map $\CsfB^{\lambda}\colon\R^{\lambda} \to \R^{\lambda}$ given by
\begin{align}
\label{eq:Burge}
\CsfB^{\lambda} &:= \Ctau_{i_n,j_n} \circ \Ctau_{i_{n-1},j_{n-1}} \circ \dots \circ \Ctau_{i_1,j_1} \, ,
\end{align}
where $((i_1,j_1),\dots,(i_n,j_n))$ is \emph{any} sequence of distinct boxes of $\lambda$ such that, for all $1\leq k\leq n$, $\lambda^{(k)} := \{(i_1,j_1),\dots, (i_k,j_k)\}$ is a Young diagram, and $\lambda^{(n)} = \lambda$.
As in RSK, even though $\CsfB^{\lambda}$ is defined for $\lambda$-shaped arrays, it may also be applied to $\mu$-shaped arrays, for any $\mu \supseteq \lambda$, by acting on the $\lambda$ part of the diagram and leaving all entries indexed by $\mu/ \lambda$ unchanged; in such a case, we do \emph{not} use the simplified notation $\CsfB$, to avoid ambiguity.

Finally, we may rephrase the definition~\eqref{eq:Burge} of the Burge correspondence in a recursive fashion: if $(m,n)$ is a corner box of $\lambda$, then
\[
\CsfB^{\lambda} =
\Ctau_{m,n} \circ \CsfB^{\lambda\setminus\{(m,n)\}} \, ,
\]
where by convention $ \sfB^{\emptyset} = \id$.

\vskip 2mm
\noindent
{\bfseries The Sch\"utzenberger involution}
For the sake of conciseness, we do not discuss the classical combinatorial construction of this correspondence, but we rather provide its piecewise linear description straightaway.
Recalling the definition of the diagonal maps $\Csigma_{k,l}$'s from~\eqref{eq:sigma_comb}, let us define the map $\CsfS = \CsfS^{m\times n} \colon \R^{m\times n} \to \R^{m\times n}$, acting on $m\times n$ matrices, by
\begin{equation}
\label{eq:Schutz}
\CsfS^{m\times n} := \Csigma_{m,1} \circ (\Csigma_{m,2} \circ \Csigma_{m,1}) \circ (\Csigma_{m,3} \circ \Csigma_{m,2} \circ \Csigma_{m,1}) \circ \dots \circ (\Csigma_{m,n-1} \circ \dots \circ \Csigma_{m,1}) \, .
\end{equation}
It is easy to see that this is an involution, using the commutative properties of the diagonal maps $\Csigma_{k,l}$ and the fact that each of them is an involution.

Denoting by $\lambda'$ the conjugate partition of a partition $\lambda$, define the \textbf{transposition map} $\sfT: \R^{\lambda} \to \R^{\lambda'}$, $\bm{w} \mapsto \bm{w}^{\sfT}$ by setting $w^{\sfT}_{i,j} := w_{j,i}$ for all $(i,j)\in\lambda'$.
By definition of the $\Csigma_{k,l}$'s, $\CsfS$ only acts on the triangular/trapezoidal lower part\footnote{Again we refer to~\eqref{eq:uvt}-\eqref{eq:RSKoutputMatrix} for the notation of lower and upper part.} $\bm{u}$ of the input matrix $\bm{t} = (\bm{u} \backslash \bm{v})$, preserving the $(n-m)$-th diagonal.
Calling $\bm{u}^{\CsfS}$ the lower part of $\CsfS(\bm{u} \backslash \bm{v})$, we thus have that $\CsfS(\bm{u} \backslash \bm{v}) = (\bm{u}^{\CsfS} \backslash \bm{v})$.
Likewise, $\sfT\CsfS\sfT$ only acts on the upper part $\bm{v}$, replacing it with another triangular/trapezoidal array that we call $\bm{v}^{\CsfS}$: we thus have that $\sfT\CsfS\sfT(\bm{u} \backslash \bm{v}) = (\bm{u} \backslash \bm{v}^{\CsfS})$.

In the case of $\bm{u}$ being a Gelfand-Tsetlin pattern (equivalently, its corresponding Young tableau), the maps $\bm{u} \mapsto \bm{u}^{\CsfS}$ and $\bm{v} \mapsto \bm{v}^{\CsfS}$ coincide with the classical Sch\"utzenberger involution defined via \emph{jeu de taquin} operations -- see e.g.~\cite{fulton97, stanley99}.
By extension, we will therefore refer to the maps $\CsfS$ and $\sfT\CsfS\sfT$ on $m\times n$ real matrices as the \textbf{Sch\"utzenberger involution} on the upper and lower part, respectively.

\vskip 2mm

Let us now define the involutions that reverse, respectively, the rows and the columns of an $m\times n$ matrix:
\begin{align}
\label{eq:rowReversion}
&\sfR:\R^{m\times n} \to \R^{m\times n} \, ,
&&\bm{w} \mapsto \bm{w}^{\sfR} \, ,
&& w^\sfR_{i,j} = w_{m-i+1,j} \, , \\
&\sfC:\R^{m\times n} \to \R^{m\times n} \, ,
&&\bm{w} \mapsto \bm{w}^{\sfC} \, ,
&& w^\sfC_{i,j} = w_{i, n-j+1} \, ,\label{eq:rowCeversion}
\end{align}
for $1\leq i\leq m$ and $1\leq j\leq n$.

The following theorems relate the RSK, Burge, and Sch\"utzenberger correspondences through row and/or column reversion of the input matrix.
As we will see in the next section, they all admit a geometric lifting.

\begin{theorem}[{\cite[Appendix~A.1]{fulton97}}]
\label{thm:RSK-Schutz}
\begin{samepage}
Let $\bm{w}\in\R^{m\times n}$.
The following diagram commutes:

\begin{center}
\begin{tikzpicture}
\node (w) at (0,0) {$\bm{w}$};
\node (uv) at (2.2,0) {$(\bm{u}\backslash \bm{v})$};
\node (wRC) at (0,-1.6) {$\bm{w}^{\sfR\sfC}$};
\node (uvS) at (2.2,-1.6) {$(\bm{u}^{\CsfS} \backslash \bm{v}^{\CsfS})$};
\draw[|->] (w) -- (uv) node[midway,above] {$\scriptstyle \CsfK$};
\draw[|->] (w) -- (wRC) node[midway,right] {$\scriptstyle \sfR\sfC$};
\draw[|->] (uv) -- (uvS) node[midway,right] {$\scriptstyle \sfT \CsfS \sfT \CsfS$};
\draw[|->] (wRC) -- (uvS) node[midway,above] {$\scriptstyle \CsfK$};
\end{tikzpicture}
\end{center}
\end{samepage}
\end{theorem}

\begin{theorem}[{\cite[Appendix~A.4.1]{fulton97}}]
\label{thm:RSK-Burge-Schutz}
\begin{samepage}
Let $\bm{w}\in\R^{m\times n}$.
The following diagrams commute:
\begin{center}
\begin{tikzpicture}
\node (w) at (0,0) {$\bm{w}$};
\node (uv) at (2.2,0) {$(\bm{u}\backslash \bm{v})$};
\node (wC) at (0,-1.6) {$\bm{w}^{\sfC}$};
\node (uvS) at (2.2,-1.6) {$(\bm{u}^{\CsfS} \backslash \bm{v})$};
\draw[|->] (w) -- (uv) node[midway,above] {$\scriptstyle \CsfK$};
\draw[|->] (w) -- (wC) node[midway,right] {$\scriptstyle \sfC$};
\draw[|->] (uv) -- (uvS) node[midway,right] {$\scriptstyle \CsfS$};
\draw[|->] (wC) -- (uvS) node[midway,above] {$\scriptstyle \CsfB$};

\begin{scope}[xshift = 5.5cm]
\node (w) at (0,0) {$\bm{w}$};
\node (uv) at (2.2,0) {$(\bm{u}\backslash \bm{v})$};
\node (wR) at (0,-1.6) {$\bm{w}^{\sfR}$};
\node (uvS) at (2.2,-1.6) {$(\bm{u} \backslash \bm{v}^{\CsfS})$};
\draw[|->] (w) -- (uv) node[midway,above] {$\scriptstyle \CsfK$};
\draw[|->] (w) -- (wR) node[midway,right] {$\scriptstyle \sfR$};
\draw[|->] (uv) -- (uvS) node[midway,right] {$\scriptstyle \sfT \CsfS \sfT$};
\draw[|->] (wR) -- (uvS) node[midway,above] {$\scriptstyle \CsfB$};
\end{scope}
\end{tikzpicture}
\end{center}
\end{samepage}
\end{theorem}

The above theorems are usually stated in the classical combinatorial context of input matrices with non-negative integer entries~\cite{fulton97}.
The even more special case of permutation matrices corresponds to the following fact: column inserting the elements $\sigma(n),\dots,\sigma(1)$ of a permutation $\sigma$ in reverse order gives the same $P$-tableau as row inserting $\sigma(1),\dots,\sigma(n)$ in the standard order, and the Sch\"utzenberger dual of the $Q$-tableau.

\section{Geometric Burge, RSK, and Sch\"utzenberger correspondences}
\label{sec:geometricCorrespondences}

In this section we perform the geometric lifting of the piecewise linear bijections introduced in Section~\ref{sec:combinatorialCorrespondences}: namely, we formally replace the ``tropical operations'' $(\vee, \wedge, +, -)$ with the ``usual four operations'' $(+,-,\times, \div )$.
This will lead to the definition of the corresponding birational maps on polygonal arrays with positive real entries, in terms of local maps on the entries.
The geometric RSK correspondence has been first studied in~\cite{kirillov01, noumiYamada04}, but our description in terms of local maps is due to~\cite{oConnellSeppalainenZygouras14} (in the case of rectangular input arrays); the geometric Sch\"utzenberger involution has been discussed in~\cite{kirillov01, noumiYamada04}; finally, to the best of our knowledge, the geometric Burge correspondence has not been considered before.
Notice that some confusion might arise from the fact that, in some of the above references, the \emph{geometric lifting} is called \emph{tropical image} instead.
We will also prove the geometric analog of Theorems~\ref{thm:RSK-Schutz} and~\ref{thm:RSK-Burge-Schutz}, a result that links together all three geometric correspondences.

For a given Young diagram $\lambda$, let $\R_{>0}^{\lambda}$ be the set of $\lambda$-shaped arrays of positive real numbers.
For $(i,j)\in\lambda$, we define $\sfa_{i,j}, \sfb_{i,j}, \sfc_{i,j} \colon \R_{>0}^{\lambda} \to \R_{>0}^{\lambda}$ as the \textbf{local maps} that act on $\bm{w}\in \R_{>0}^{\lambda}$ by only modifying $w_{i,j}$ according to the following rules:
\begin{align}
\label{eq:a}
&\sfa_{i,j}\colon w_{i,j} \longmapsto \frac{1}{w_{i,j}}(w_{i-1,j} + w_{i,j-1}) \left(\frac{1}{w_{i+1,j}} + \frac{1}{w_{i,j+1}} \right)^{-1} \, , \\
\label{eq:b}
&\sfb_{i,j}\colon w_{i,j} \longmapsto \frac{1}{w_{i,j}}(w_{i-1,j} + w_{i,j-1}) w_{i,j+1} \, , \\
\label{eq:c}
&\sfc_{i,j}\colon w_{i,j} \longmapsto w_{i,j} (w_{i-1,j} + w_{i,j-1}) \, .
\end{align}
For two distinct indices $(i,j),(k,l)\in\lambda$, we also define $\sfd^{k,l}_{i,j} \colon \R_{>0}^{\lambda} \to \R_{>0}^{\lambda}$ as the local map that acts on an array $\bm{w}\in \R_{>0}^{\lambda}$ by only modifying $w_{i,j}$ and $w_{k,l}$ as follows:
\begin{equation}
\label{eq:d}
\sfd^{k,l}_{i,j}\colon
\begin{cases}
w_{i,j} \longmapsto \left( \dfrac{1}{w_{i,j}} + \dfrac{1}{w_{k,l} (w_{i-1,j} + w_{i,j-1})} \right)^{-1} \, , \\
w_{k,l} \longmapsto \left( \dfrac{w_{k,l}(w_{i-1,j} + w_{i,j-1})}{w_{i,j}^2} + \dfrac{1}{w_{i,j}} \right) \left(\dfrac{1}{w_{i+1,j}} + \dfrac{1}{w_{i,j+1}} \right)^{-1} \, .
\end{cases}
\end{equation}
For $i=1$ and/or $j=1$, the values of $w_{i-1,j}$ and $w_{i,j-1}$ are determined by the following convention: $w_{0,1}=w_{1,0}=1/2$ and $w_{0,k}=w_{k,0}=0$ for all $k>1$.
For $\sfa_{i,j}$, $\sfb_{i,j}$, and $\sfd_{i,j}^{k,l}$ to be well-defined, $(i+1,j)$ and/or $(i,j+1)$ must be boxes of $\lambda$.
It will also be useful to define the map $\sfe_{i,j}^{k,l}$, which acts on a $\lambda$-shaped array $\bm{w}$, with $(i,j),(k,l)\in\lambda$, by exchanging $w_{i,j}$ with $w_{k,l}$:
\begin{equation}
\label{eq:e}
\sfe^{k,l}_{i,j}\colon
\begin{cases}
w_{i,j} \longmapsto w_{k,l} \, , \\
w_{k,l} \longmapsto w_{i,j} \, .
\end{cases}
\end{equation}
All these local maps are bijective, but only $\sfa_{i,j}$, $\sfb_{i,j}$, and $\sfe_{i,j}^{k,l}$ are involutions.
As in the ``tropical'' case, they all satisfy obvious commutative properties, due to their local action on the entries of the input array.
For example, local maps of type~\eqref{eq:a}-\eqref{eq:c} commute whenever the subscripts are not nearest neighbors in $\N^2$.

Let us now define, for all $(k,l)\in\lambda$, the following \textbf{diagonal maps}:
\begin{align}
\label{eq:rho}
\rho_{k,l} &:= \sfa_{k-h+1,l-h+1} \circ \sfa_{k-h+2,l-h+2} \circ \dots \circ \sfa_{k-1,l-1} \circ \sfc_{k,l} \, , \\
\label{eq:sigma}
\sigma_{k,l} &:= \sfa_{k-h+1,l-h+1} \circ \sfa_{k-h+2,l-h+2} \circ \dots \circ \sfa_{k-1,l-1} \circ \sfb_{k,l} \, , \\
\label{eq:tau}
\tau_{k,l} &:= \sfc_{k,l} \circ \sfd^{k,l}_{k-1,l-1} \circ \dots \circ \sfd^{k,l}_{k-h+2,l-h+2} \circ \sfd^{k,l}_{k-h+1,l-h+1} \, ,
\end{align}
where $h:=k \wedge l$.
All these maps are bijective, and $\sigma_{k,l}$ is also an involution.
Any two diagonal maps commute if they act on diagonals that are not the same nor neighboring.

We can now define the \textbf{geometric Robinson-Schensted-Knuth (RSK) correspondence} $\sfK = \sfK^{\lambda}$ and the \textbf{geometric Burge correspondence} $\sfB = \sfB^{\lambda}$ as the bijections $\R_{>0}^{\lambda} \to \R_{>0}^{\lambda}$ given by
\begin{align}
\label{eq:gRSK}
\sfK &:= \rho_{i_n,j_n} \circ \rho_{i_{n-1},j_{n-1}} \circ \dots \circ \rho_{i_1,j_1} \, , \\
\label{eq:gBurge}
\sfB &:= \tau_{i_n,j_n} \circ \tau_{i_{n-1},j_{n-1}} \circ \dots \circ \tau_{i_1,j_1} \, .
\end{align}
As in~\eqref{eq:RSK} and~\eqref{eq:Burge}, here $((i_1,j_1),\dots,(i_n,j_n))$ is any sequence of distinct boxes such that, for all $1\leq k\leq n$, $\lambda^{(k)} := \{(i_1,j_1),\dots, (i_k,j_k)\}$ is a Young diagram, and $\lambda^{(n)} = \lambda$.
We will be mostly using the following equivalent recursive definition of $\sfK$ and $\sfB$:
\begin{align}
\label{eq:gRSK-gBurge_recursive}
\sfK^{\lambda} =
\rho_{m,n} \circ \sfK^{\lambda\setminus\{(m,n)\}} \quad\qquad\text{and}\qquad\quad
\sfB^{\lambda} =
\tau_{m,n} \circ \sfB^{\lambda\setminus\{(m,n)\}} \, ,
\end{align}
for any corner box $(m,n)$ of $\lambda$, where by convention $\sfK^{\emptyset} = \sfB^{\emptyset} = \id$.

Recall that we denote by $\sfT$ the map that acts on a Young-diagram-shaped array by transposing it in the usual way.
As in the ``tropical'' case, it is easy to see that the geometric RSK and Burge correspondences satisfy a symmetry property: $\sfK(\bm{w}^{\sfT}) = \sfK(\bm{w})^{\sfT}$ and $\sfB(\bm{w}^{\sfT}) = \sfB(\bm{w})^{\sfT}$ for all $\bm{w}\in\R_{>0}^{\lambda}$ -- see Proposition~\ref{prop:Burge_symmetric} for a formal statement and proof in the Burge case.

We next define the \textbf{geometric Sch\"utzenberger involution} $\sfS = \sfS^{m\times n} \colon \R_{>0}^{m\times n} \to \R_{>0}^{m\times n}$ by
\begin{equation}
\label{eq:gSchutz}
\sfS^{m\times n} := \sigma_{m,1} \circ (\sigma_{m,2} \circ \sigma_{m,1}) \circ (\sigma_{m,3} \circ \sigma_{m,2} \circ \sigma_{m,1}) \circ \dots \circ (\sigma_{m,n-1} \circ \dots \circ \sigma_{m,1}) \, .
\end{equation}
Similarly to Section~\ref{sec:combinatorialCorrespondences}, we write $\bm{t}=(\bm{u}\backslash \bm{v})$ for a matrix $\bm{t}$ with ``lower part'' $\bm{u}$ and ``upper part'' $\bm{v}$ (i.e.\ the parts that $\bm{t}$ is divided into by the diagonal that contains the bottom-right corner of the matrix) -- see~\eqref{eq:uv}-\eqref{eq:uvt}-\eqref{eq:RSKoutputMatrix}.
We have that $\sfS$ (respectively, $\sfT\sfS\sfT$) acts on $\bm{t}$ by modifying the lower part $\bm{u}$ (respectively, the upper part $\bm{v}$) only, and preserving the $(n-m)$-th diagonal.
Therefore, if $\bm{u}^{\sfS}$ is the lower part of $\sfS(\bm{u} \backslash \bm{v})$ and $\bm{v}^{\sfS}$ is the upper part of $\sfT\sfS\sfT(\bm{u} \backslash \bm{v})$, then we have that $\sfS(\bm{u} \backslash \bm{v}) = (\bm{u}^{\sfS} \backslash \bm{v})$ and $\sfT\sfS\sfT(\bm{u} \backslash \bm{v}) = (\bm{u} \backslash \bm{v}^{\sfS})$.
The maps $\bm{u} \mapsto \bm{u}^{\sfS}$ and $\bm{v} \mapsto \bm{v}^{\sfS}$ can be regarded as the geometric lifting of (the generalization of) the Sch\"utzenberger involution studied in~\cite{berensteinKirillov96}.

The relation between geometric RSK and geometric Sch\"utzenberger involution goes through both row and column reversion of the input matrix, as stated in the following theorem, which is the geometric analog of Theorem~\ref{thm:RSK-Schutz}.
As in~\eqref{eq:rowReversion}-\eqref{eq:rowCeversion}, $\sfR$ and $\sfC$ denote the maps that reverse, respectively, the rows and the columns of a matrix.
\begin{theorem}[{\cite[Section~4.5]{kirillov01}}, {\cite[Sections~2.4, 3.1]{noumiYamada04}}]
\label{thm:gRSK-Schutz}
Let $\bm{w} \in\R_{>0}^{m\times n}$.
The following diagram commutes:

\begin{center}
\begin{tikzpicture}
\node (w) at (0,0) {$\bm{w}$};
\node (uv) at (2.2,0) {$(\bm{u}\backslash \bm{v})$};
\node (wRC) at (0,-1.6) {$\bm{w}^{\sfR\sfC}$};
\node (uvS) at (2.2,-1.6) {$(\bm{u}^{\sfS} \backslash \bm{v}^{\sfS})$};
\draw[|->] (w) -- (uv) node[midway,above] {$\scriptstyle \sfK$};
\draw[|->] (w) -- (wRC) node[midway,right] {$\scriptstyle \sfR\sfC$};
\draw[|->] (uv) -- (uvS) node[midway,right] {$\scriptstyle \sfT \sfS \sfT \sfS$};
\draw[|->] (wRC) -- (uvS) node[midway,above] {$\scriptstyle \sfK$};
\end{tikzpicture}
\end{center}
\end{theorem}

We will now prove a stronger and fundamental result that represents the geometric lifting of Theorem~\ref{thm:RSK-Burge-Schutz}.
It connects the geometric RSK, Burge, and Sch\"utzenberger correspondences through either column reversion or row reversion of the input matrix.

\begin{theorem}
\label{thm:gRSK-Burge-Schutz}
\begin{samepage}
Let $\bm{w}\in\R_{>0}^{m\times n}$.
The following diagrams commute:
\begin{center}
\begin{tikzpicture}
\node (w) at (0,0) {$\bm{w}$};
\node (uv) at (2.2,0) {$(\bm{u}\backslash \bm{v})$};
\node (wC) at (0,-1.6) {$\bm{w}^{\sfC}$};
\node (uvS) at (2.2,-1.6) {$(\bm{u}^{\sfS} \backslash \bm{v})$};
\draw[|->] (w) -- (uv) node[midway,above] {$\scriptstyle \sfK$};
\draw[|->] (w) -- (wC) node[midway,right] {$\scriptstyle \sfC$};
\draw[|->] (uv) -- (uvS) node[midway,right] {$\scriptstyle \sfS$};
\draw[|->] (wC) -- (uvS) node[midway,above] {$\scriptstyle \sfB$};

\begin{scope}[xshift = 5.5cm]
\node (w) at (0,0) {$\bm{w}$};
\node (uv) at (2.2,0) {$(\bm{u}\backslash \bm{v})$};
\node (wR) at (0,-1.6) {$\bm{w}^{\sfR}$};
\node (uvS) at (2.2,-1.6) {$(\bm{u} \backslash \bm{v}^{\sfS})$};
\draw[|->] (w) -- (uv) node[midway,above] {$\scriptstyle \sfK$};
\draw[|->] (w) -- (wR) node[midway,right] {$\scriptstyle \sfR$};
\draw[|->] (uv) -- (uvS) node[midway,right] {$\scriptstyle \sfT \sfS \sfT$};
\draw[|->] (wR) -- (uvS) node[midway,above] {$\scriptstyle \sfB$};
\end{scope}
\end{tikzpicture}
\end{center}
\end{samepage}
\end{theorem}

Notice that Theorem~\ref{thm:gRSK-Schutz} can be indeed derived as a straightforward consequence of Theorem~\ref{thm:gRSK-Burge-Schutz}.
To prove the latter result we need the following proposition, whose proof is quite involved and is postponed to the appendix.
\begin{proposition}
\label{prop:compDiagMaps}
If $(p,q),(p,q+1),(p-1,q) \in \lambda$, then the following relation between maps acting on $\lambda$-shaped arrays holds:
\begin{equation}
\label{eq:compDiagMaps}
\sigma_{p,q} \rho_{p,q+1} \tau_{p,q}
= \tau_{p,q+1} \rho_{p,q} \sigma_{p-1,q} \sfe_{p,q}^{p,q+1} \, .
\end{equation}
\end{proposition}

The latter can be seen as a structural `commutation relation' between the diagonal maps involved in the geometric RSK, Burge, and Sch\"utzenberger bijections.
Following the same lines of our proofs (or, alternatively, using a tropical limit procedure as in~\cite[\S~1.1.3]{bisi18}), one can argue that an analogous commutation relation holds for the corresponding piece-wise linear maps discussed in Section~\ref{sec:combinatorialCorrespondences}.
However, we are not aware of any combinatorial version of Proposition~\ref{prop:compDiagMaps} in the literature and, moreover, this appears to be a somewhat non-trivial identity, even in the combinatorial setting.
Notice that the combinatorial analog of Theorem~\ref{thm:gRSK-Burge-Schutz}, i.e.\ Theorem~\ref{thm:RSK-Burge-Schutz}, is classically proven without resorting to the piece-wise linear formulation of the combinatorial correspondences.

\begin{proof}[Proof of Theorem~\ref{thm:gRSK-Burge-Schutz}]
The second relation $\sfT \sfS \sfT \sfK = \sfB \sfR$ just follows from the first relation $\sfB \sfC = \sfS \sfK$ as well as the basic properties of the maps involved.
Namely, assuming that $\sfB \sfC = \sfS \sfK$ holds true, recalling that both $\sfK$ and $\sfB$ commute with the transposition, and using the trivial fact that $\sfT \sfC \sfT = \sfR$, we obtain:
\[
\sfT \sfS \sfT \sfK
= \sfT \sfS \sfK \sfT
= \sfT \sfB \sfC \sfT
= \sfB \sfT \sfC \sfT
= \sfB \sfR \, .
\]

We are then left to prove that $\sfB \sfC = \sfS \sfK$ as maps $\R_{>0}^{m\times n} \to \R_{>0}^{m\times n}$; to do so, we will apply the induction principle several times.
Let us first fix any $n\geq 1$ and proceed by induction on $m$, i.e.\ the number of rows of the matrices.
For $m=1$ and a $1\times n$ matrix $\bm{w} = \begin{pmatrix} w_1 &w_2 &\dots &w_{n-1} &w_n \end{pmatrix}$, by definition we have that
\[
\sfK(\bm{w})
= \sfB(\bm{w})
= \begin{pmatrix} w_1 &w_1 w_2 &\dots &\prod_{k=1}^{n-1} w_k &\prod_{k=1}^n w_k \end{pmatrix} \, .
\]
It is also easy to check that
\[
\sfS(\bm{w})
= \begin{pmatrix} \dfrac{w_n}{w_{n-1}} &\dfrac{w_n}{w_{n-2}} &\dots &\dfrac{w_n}{w_1} &w_n \end{pmatrix} \, .
\]
It follows that
\[
\begin{split}
\sfB \sfC (\bm{w})
&= \sfB \begin{pmatrix} w_n &w_{n-1} &\dots &w_2 &w_1 \end{pmatrix} \\
&= \begin{pmatrix} w_n &w_{n-1} w_n &\dots &\prod_{k=2}^n w_k &\prod_{k=1}^n w_k \end{pmatrix} \, , \\
&= \sfS \begin{pmatrix} w_1 &w_1 w_2 &\dots &\prod_{k=1}^{n-1} w_k &\prod_{k=1}^n w_k \end{pmatrix}
= \sfS \sfK (\bm{w}) \, ,
\end{split}
\]
thus proving that $\sfB \sfC = \sfS \sfK$ for $1\times n$ matrices.

Let us now suppose by induction that for a given $m>1$ the statement is true in the case of $(m-1)\times n$ matrices for all $n\geq 1$, and prove the statement in the case of $m\times n$ matrices for all $n\geq 1$.
Let $\bm{x} = (x_{i,j})_{1\leq i\leq m-1, 1\leq j\leq n} \in \R_{>0}^{(m-1)\times n}$, $\bm{y} = (y_1,\dots,y_n) \in \R_{>0}^{1\times n}$, and
\[
\bm{w} := \begin{pmatrix}
\bm{x} \\
\bm{y}
\end{pmatrix}
= \begin{pmatrix}
x_{1,1} &\dots &x_{1,n} \\
\vdots &\ddots &\vdots \\
x_{m-1,1} &\dots &x_{m-1,n} \\
y_1 &\dots &y_n
\end{pmatrix} \, .
\]
By definition, the geometric Burge correspondence on $\bm{w}$ can be obtained by applying maps $\tau_{k,l}$'s first for all $1\leq k\leq m-1$ and $1\leq l\leq n$, and subsequently for $k=m$ and $1\leq l\leq n$.
Therefore,
\[
\sfB \sfC(\bm{w}) 
= \tau_{m,n} \cdots \tau_{m,1} \begin{pmatrix}
\sfB \sfC(\bm{x}) \\
\sfC(\bm{y})
\end{pmatrix}
= \tau_{m,n} \cdots \tau_{m,1} \begin{pmatrix}
\sfS \sfK(\bm{x}) \\
\sfC(\bm{y})
\end{pmatrix} \, ,
\]
where in the latter equality we have applied the induction hypothesis on $\bm{x}$.
The same reasoning holds for the geometric RSK as a composition of maps $\rho_{k,l}$'s:
\[
\sfS \sfK (\bm{w})
= \sfS \, \rho_{m,n} \cdots \rho_{m,1} \begin{pmatrix}
\sfK (\bm{x}) \\
\bm{y}
\end{pmatrix} \, .
\]
Since $\sfK$ is invertible, to conclude that $\sfB \sfC = \sfS \sfK$ on $m\times n$ matrices it suffices to show that
\begin{equation}
\label{eq:gRSK-Burge-Schutz_induction}
\tau_{m,n} \cdots \tau_{m,1} \begin{pmatrix}
\sfS (\bm{x}) \\
\sfC(\bm{y})
\end{pmatrix}
= \sfS \, \rho_{m,n} \cdots \rho_{m,1} \begin{pmatrix}
\bm{x} \\
\bm{y}
\end{pmatrix} \, .
\end{equation}

We are then left to prove~\eqref{eq:gRSK-Burge-Schutz_induction} for all $m> 1$ and $n\geq 1$.
We will now fix $m$ and proceed by induction on $n$.
The statement for $n=1$ follows from the fact that $\tau_{m,1} = \sfc_{m,1} = \rho_{m,1}$ and $\sfS^{k\times 1} = \sfC^{k\times 1} = \id^{k\times 1}$ for any $k\geq 1$.
We will now show that, for any given $N>1$, if~\eqref{eq:gRSK-Burge-Schutz_induction} holds for $n=N-1$, then it also holds for $n=N$.
For $n=N$, the left-hand side of~\eqref{eq:gRSK-Burge-Schutz_induction} reads as
\[
\begin{split}
&\tau_{m,N} \cdots \tau_{m,1} \begin{pmatrix}
\sfS(\bm{x}) \\
\sfC(\bm{y})
\end{pmatrix}
= \tau_{m,N}
(\tau_{m,N-1} \cdots \tau_{m,1})
\begin{pmatrix}
\sfS^{(m-1)\times (N-1)} (\sigma_{m-1,N-1} \cdots \sigma_{m-1,1}) (\bm{x}) \\
\begin{matrix}
y_N &y_{N-1} &\cdots &y_2 &y_1
\end{matrix}
\end{pmatrix} \\
& \quad = \tau_{m,N} \sfS^{m \times (N-1)}
(\rho_{m,N-1} \cdots \rho_{m,1}) \begin{pmatrix}
(\sigma_{m-1,N-1} \cdots \sigma_{m-1,1}) (\bm{x}) \\
\begin{matrix}
y_2 &\cdots &y_{N-1} &y_N &y_1
\end{matrix}
\end{pmatrix} \\
& \quad = \tau_{m,N} \sfS^{m \times (N-1)}
(\rho_{m,N-1} \cdots \rho_{m,1})
(\sfe_{m,N-1}^{m,N} \cdots \sfe_{m,1}^{m,2})
\begin{pmatrix}
(\sigma_{m-1,N-1} \cdots \sigma_{m-1,1}) (\bm{x}) \\
\begin{matrix}
y_1 &y_2 &\cdots &y_{N-1} &y_N
\end{matrix}
\end{pmatrix} \\
& \quad = \tau_{m,N} \sfS^{m \times (N-1)}
(\rho_{m,N-1} \cdots \rho_{m,1})
(\sigma_{m-1,N-1} \cdots \sigma_{m-1,1})
(\sfe_{m,N-1}^{m,N} \cdots \sfe_{m,1}^{m,2})
\begin{pmatrix}
\bm{x} \\
\bm{y}
\end{pmatrix} \\
& \quad = \sfS^{m \times (N-1)} \tau_{m,N}
(\rho_{m,N-1} \sigma_{m-1,N-1} \sfe_{m,N-1}^{m,N})
\cdots (\rho_{m,1} \sigma_{m-1,1} \sfe_{m,1}^{m,2})
\begin{pmatrix}
\bm{x} \\
\bm{y}
\end{pmatrix} \, .
\end{split}
\]
For the above equalities we have used, in order: the recursive definition of $\sfS$; the induction hypothesis, i.e.~\eqref{eq:gRSK-Burge-Schutz_induction} for $n=N-1$; the definition of the exchange operator $\sfe_{i,j}^{k,l}$; finally, the commutative properties of local and diagonal maps.
On the other hand, for $n=N$, the right-hand side of~\eqref{eq:gRSK-Burge-Schutz_induction} reads as
\[
\begin{split}
\sfS \, \rho_{m,N} \cdots \rho_{m,1} \begin{pmatrix}
\bm{x} \\
\bm{y}
\end{pmatrix}
&= \sfS^{m\times (N-1)} (\sigma_{m,N-1} \cdots \sigma_{m,1}) (\rho_{m,N} \cdots \rho_{m,2} \rho_{m,1})
\begin{pmatrix}
\bm{x} \\
\bm{y}
\end{pmatrix} \\
& = \sfS^{m\times (N-1)} (\sigma_{m,N-1} \rho_{m,N}) \cdots (\sigma_{m,1} \rho_{m,2}) \rho_{m,1}
\begin{pmatrix}
\bm{x} \\
\bm{y}
\end{pmatrix} \, ,
\end{split}
\]
again by definition of $\sfS$ and the commutative properties.
To conclude that~\eqref{eq:gRSK-Burge-Schutz_induction} holds for $n=N$, it thus remains to show that
\begin{equation}
\label{eq:gRSK-Burge-Schutz_2}
\begin{split}
&\tau_{m,N}
(\rho_{m,N-1} \sigma_{m-1,N-1} \sfe_{m,N-1}^{m,N})
\cdots (\rho_{m,1} \sigma_{m-1,1} \sfe_{m,1}^{m,2}) \\
&\qquad = (\sigma_{m,N-1} \rho_{m,N}) \cdots (\sigma_{m,1} \rho_{m,2}) \rho_{m,1} \, .
\end{split}
\end{equation}
In turn, the latter readily follows from $N-1$ iterative applications of Proposition~\ref{prop:compDiagMaps} together with the already noticed fact that $\tau_{m,1} = \rho_{m,1}$.
\end{proof}

\section{Properties of the geometric Burge correspondence}
\label{sec:Burge_properties}

In this section we prove the volume preserving property and other useful properties of the geometric Burge correspondence.
Such properties will follow either directly from the definition via local maps given in Section~\ref{sec:geometricCorrespondences} or from Theorem~\ref{thm:gRSK-Burge-Schutz}, as a consequence of the analogous properties of the geometric RSK correspondence~\cite{oConnellSeppalainenZygouras14}.

For the geometric RSK correspondence, it is known that the product of the last $k$ entries of a diagonal in the output array can be expressed in terms of the input array as a ``partition function'' on $k$ non-intersecting directed lattice paths.
\begin{proposition}[\cite{noumiYamada04, oConnellSeppalainenZygouras14}]
\label{prop:non-intersPaths_gRSK}
Let $\bm{w}\in \R_{>0}^{m\times n}$ and $\bm{t}:= \sfK(\bm{w})$.
For all $1\leq k\leq m \wedge n$ we have that
\begin{equation}
\label{eq:non-intersPaths_gRSK}
t_{m,n} t_{m-1,n-1} \cdots t_{m-k+1, n-k+1}
= \sum_{(\pi_1,\dots,\pi_k) \in \Pi^{(k)}_{m,n}} \prod_{(i,j)\in \pi_1 \cup \cdots \cup \pi_k} w_{i,j} \, ,
\end{equation}
where $\Pi^{(k)}_{m,n}$ is the set of $k$-tuples of non-intersecting directed lattice paths in $\N^2$ starting at $(1,1)$, $(1,2)$, \dots, $(1,k)$ and ending at $(m,n-k+1)$, $(m,n-k+2)$, \dots, $(m,n)$ respectively.
\end{proposition}

The geometric Burge correspondence has a similar property, where the non-intersecting paths go in the north-east direction instead of south-east.
This fact is proven in the next proposition.
We state the result for generic Young-diagram-shaped arrays, and specialize it to the ``extreme'' cases, which are of particular interest.
First, denote the product of all the entries on the $k$-th diagonal of $\bm{t}\in\R_{>0}^{\lambda}$ by
\begin{equation}
\label{eq:diagProd}
P_k(\bm{t}) := \prod_{j-i = k} t_{i,j} \, .
\end{equation}

\begin{proposition}
\label{prop:non-intersPaths}
Let $\bm{w}\in \R_{>0}^{\lambda}$ and $\bm{t}:= \sfB(\bm{w})$.
If $(m,n)$ is a border box of $\lambda$, then for all $1\leq k\leq m \wedge n$ we have that
\begin{equation}
\label{eq:non-intersPaths}
t_{m,n} t_{m-1,n-1} \cdots t_{m-k+1, n-k+1}
= \sum_{(\pi_1,\dots,\pi_k) \in \Pi^{*(k)}_{m,n}} \prod_{(i,j)\in \pi_1 \cup \cdots \cup \pi_k} w_{i,j} \, ,
\end{equation}
where $\Pi^{*(k)}_{m,n}$ is the set of $k$-tuples of non-intersecting directed lattice paths in $\N^2$ starting at $(m,1)$, $(m,2)$, \dots, $(m,k)$ and ending at $(1,n-k+1)$, $(1,n-k+2)$, \dots, $(1,n)$ respectively.
In particular, for $k=1$,
\begin{equation}
\label{eq:non-intersPaths_partitionFn}
t_{m,n} 
= \sum_{\pi \in \Pi^{(1)}_{m,n}} \prod_{(i,j)\in \pi} w_{i,j} \, ,
\end{equation}
and, for $k=m\wedge n$,
\begin{equation}
\label{eq:non-intersPaths_diagProd}
P_{n-m}(\bm{t})
= \prod_{i=1}^m \prod_{j=1}^n w_{i,j} \, .
\end{equation}
\end{proposition}
\begin{proof}
Identities~\eqref{eq:non-intersPaths_partitionFn} and~\eqref{eq:non-intersPaths_diagProd} are straightforward consequences of~\eqref{eq:non-intersPaths}, hence we only need to prove the latter.
It is clear from~\eqref{eq:gBurge} that
\begin{equation}
\label{eq:gBurge2}
\bm{t}
= \sfB^{\lambda}(\bm{w})
=\tau_{i_l, j_l} \circ \cdots \circ \tau_{i_1, j_1} \circ \sfB^{m\times n} (\bm{w}) \, ,
\end{equation}
where $l= \abs{\lambda} - mn$ and $(i_1,j_1), \dots, (i_l,j_l)$ are chosen so that $(m\times n) \cup \{(i_1,j_1), \dots, (i_h, j_h)\}$ is a Young diagram for all $1\leq h\leq l$.
Since by hypothesis $(m,n)$ is the last box of $\lambda$ on the corresponding diagonal, the application of $\tau_{i_1, j_1}$, \dots, $\tau_{i_l, j_l}$ in~\eqref{eq:gBurge2} does not modify the $(n-m)$-th diagonal of $\sfB^{m\times n} (\bm{w})$.
It then suffices to prove~\eqref{eq:non-intersPaths} when $\lambda = m\times n$ is a rectangular partition, so we now restrict to this case.
By Theorem~\ref{thm:gRSK-Burge-Schutz} we have that $\bm{t} = \sfB (\bm{w}) = \tilde{\bm{t}}^{\sfS}$, where $\tilde{\bm{t}} := \sfK (\bm{w}^{\sfC})$.
By Proposition~\ref{prop:non-intersPaths_gRSK}, we have that
\[
\begin{split}
\tilde{t}_{m,n} \tilde{t}_{m-1,n-1} \cdots \tilde{t}_{m-k+1, n-k+1}
&= \sum_{(\pi_1,\dots,\pi_k) \in \Pi^{(k)}_{m,n}} \prod_{(i,j)\in \pi_1 \cup \cdots \cup \pi_k} w_{i,n-j+1} \\
&= \sum_{(\pi_1,\dots,\pi_k) \in \Pi^{*(k)}_{m,n}} \prod_{(i,j)\in \pi_1 \cup \cdots \cup \pi_k} w_{i,j} \, .
\end{split}
\]
The geometric Sch\"utzenberger involution does not modify the $(n-m)$-th diagonal of $\tilde{\bm{t}}$, hence $\tilde{t}_{m,n} = t_{m,n}$, $\tilde{t}_{m-1,n-1} = t_{m-1,n-1}$, \dots, $\tilde{t}_{m-k+1,n-k+1} = t_{m-k+1,n-k+1}$.
The above display then proves~\eqref{eq:non-intersPaths} in the case $\lambda= m\times n$.
\end{proof}

We now state another property of the geometric Burge correspondence, which expresses the sum of certain ratios of output entries as a sum of inverse input entries.

\begin{proposition}
\label{prop:sumInvWeights}
Let $\bm{w}\in \R_{>0}^{\lambda}$ and $\bm{t}:= \sfB(\bm{w})$.
We have that
\begin{align}
\label{eq:sumInvWeights_diag}
\frac{1}{t_{1,1}}
&= \sum_{i\colon (i,i)\in \lambda} \frac{1}{w_{i,i}} \, , \\
\label{eq:sumInvWeights}
\sum_{(i,j)\in \lambda} \frac{t_{i-1,j} + t_{i,j-1}}{t_{i,j}}
&= \sum_{(i,j)\in \lambda} \frac{1}{w_{i,j}} \, ,
\end{align}
with the convention that $t_{0,1}=t_{1,0}=1/2$ and $t_{0,k}=t_{k,0}=0$ for all $k>1$.
\end{proposition}
\begin{proof}
Let $(n,n)$ be the (only) border box of $\lambda$ on the main diagonal.
Then~\eqref{eq:sumInvWeights_diag} can be obtained by setting $m=n$ in Proposition~\ref{prop:non-intersPaths} and dividing both sides of equation~\eqref{eq:non-intersPaths}, taken with $k=n-1$, by the corresponding sides of~\eqref{eq:non-intersPaths_diagProd}.

To prove~\eqref{eq:sumInvWeights}, we proceed by induction on $\abs{\lambda}$.
If $\lambda$ has just the box $(1,1)$, then~\eqref{eq:sumInvWeights} follows for example from~\eqref{eq:sumInvWeights_diag}.
Assume now that $\abs{\lambda} >1$, and pick any corner box $(m,n)$ of $\lambda$.
Let $\bm{w}\in \R^{\lambda}_{>0}$ and $\bm{t}:= \sfB(\bm{w})$.
Set $\tilde{\lambda} := \lambda \setminus \{(m,n)\}$ and $\tilde{\bm{t}}:= \sfB^{\tilde{\lambda}}(\bm{w})$, so that $\bm{t} = \tau_{m,n}(\tilde{\bm{t}})$ by~\eqref{eq:gRSK-gBurge_recursive}.
By induction, the statement holds for the input array $\bm{w}$ and the output array $\tilde{\bm{t}}$ when they are both restricted to the partition $\tilde{\lambda}$.
It then suffices to prove that
\begin{equation}
\label{eq:sumInvWeights_induction}
\sum_{(i,j)\in \lambda} \frac{t_{i-1,j} + t_{i,j-1}}{t_{i,j}}
= \sum_{(i,j)\in \tilde{\lambda}} \frac{\tilde{t}_{i-1,j} + \tilde{t}_{i,j-1}}{\tilde{t}_{i,j}}
+ \frac{1}{w_{m,n}} \, .
\end{equation}
To fix the ideas, let us suppose $m\leq n$ and set $p:= n-m\geq 0$.
By~\eqref{eq:tau}, we then have that
\[
\tau_{m,n} = \sfc_{m,n} \circ \sfd^{m,n}_{m-1,n-1} \circ \dots \circ \sfd^{m,n}_{2,p+2} \circ \sfd^{m,n}_{1,p+1} \, .
\]
In particular, the arrays $\bm{t}$ and $\tilde{\bm{t}}$ only differ at the entries on the $p$-th diagonal.
Therefore, we may restrict the sums on both sides of~\eqref{eq:sumInvWeights_induction} to the terms that involve entries on the $p$-th diagonal of $\bm{t}$ and $\tilde{\bm{t}}$.
We are thus left to prove that
\begin{equation}
\label{eq:sumInvWeights_induction2}
\begin{split}
&\sum_{i=1}^{m-1} \left[ \frac{t_{i-1,p+i} + t_{i,p+i-1}}{t_{i,p+i}} + t_{i,p+i} \left(\frac{1}{t_{i+1,p+i}} + \frac{1}{t_{i,p+i+1}}\right) \right] + \frac{t_{m-1,n} + t_{m,n-1}}{t_{m,n}} \\
= \, &\sum_{i=1}^{m-1} \left[ \frac{t_{i-1,p+i} + t_{i,p+i-1}}{\tilde{t}_{i,p+i}} + \tilde{t}_{i,p+i} \left(\frac{1}{t_{i+1,p+i}} + \frac{1}{t_{i,p+i+1}}\right) \right] +\frac{1}{w_{m,n}} \, .
\end{split}
\end{equation}
To prove the latter, set $\bm{t}^{(0)} := \tilde{\bm{t}}$ and $\bm{t}^{(i)} := \sfd^{m,n}_{i,p+i}(\bm{t}^{(i-1)})$ for $1\leq i\leq m-1$, so that $\bm{t} = \sfc_{m,n}(\bm{t}^{(m-1)})$.
By definitions~\eqref{eq:d} and~\eqref{eq:c}, we thus have that
\begin{align}
\label{eq:sumInvWeights_d}
t_{i,p+i}
&= \left( \frac{1}{\tilde{t}_{i,p+i}} + \frac{1}{t^{(i-1)}_{m,n} (t_{i-1,p+i} + t_{i,p+i-1})} \right)^{-1} &&\text{for }1\leq i\leq m-1 \, , \\
\label{eq:sumInvWeights_c}
t_{m,n} &= t^{(m-1)}_{m,n}(t_{m-1,n} + t_{m,n-1}) \, .
\end{align}
Again from~\eqref{eq:d} it also follows that
\[
t^{(i)}_{m,n} = \left( \frac{t^{(i-1)}_{m,n} (t_{i-1,p+i} + t_{i,p+i-1})}{\tilde{t}_{i,p+i}^2} + \frac{1}{\tilde{t}_{i,p+i}} \right) \left(\frac{1}{t_{i+1,p+i}} + \frac{1}{t_{i,p+i+1}} \right)^{-1}
\]
for $1\leq i\leq m-1$, which is equivalent to
\begin{equation}
\label{eq:sumInvWeights_d2}
t^{(i-1)}_{m,n}
= \left[ t^{(i)}_{m,n} \left(\frac{1}{t_{i+1,p+i}} + \frac{1}{t_{i,p+i+1}} \right) - \frac{1}{\tilde{t}_{i,p+i}} \right] \frac{\tilde{t}_{i,p+i}^2}{t_{i-1,p+i} + t_{i,p+i-1}} \, .
\end{equation}
From~\eqref{eq:sumInvWeights_d} and~\eqref{eq:sumInvWeights_d2} we obtain another expression for $t_{i,p+i}$ that involves $t^{(i)}_{m,n}$ instead of $t^{(i-1)}_{m,n}$:
\begin{equation}
\label{eq:sumInvWeights_d3}
t_{i,p+i} = \tilde{t}_{i,p+i} - \frac{1}{t^{(i)}_{m,n}} \left(\frac{1}{t_{i+1,p+i}} + \frac{1}{t_{i,p+i+1}} \right)^{-1} \, .
\end{equation}
We now compute the left-hand side of~\eqref{eq:sumInvWeights_induction2} by using~\eqref{eq:sumInvWeights_d} for the first occurrence of $t_{i,p+i}$ ($1\leq i\leq m-1$), \eqref{eq:sumInvWeights_d3} for the second occurrence of $t_{i,p+i}$ ($1\leq i\leq m-1$), and~\eqref{eq:sumInvWeights_c} for the only occurrence of $t_{m,n}$:
\[
\begin{split}
&\sum_{i=1}^{m-1} \left[ \frac{t_{i-1,p+i} + t_{i,p+i-1}}{t_{i,p+i}} + t_{i,p+i} \left(\frac{1}{t_{i+1,p+i}} + \frac{1}{t_{i,p+i+1}}\right) \right] + \frac{t_{m-1,n} + t_{m,n-1}}{t_{m,n}} \\
= \, & \sum_{i=1}^{m-1} \left[ \frac{t_{i-1,p+i} + t_{i,p+i-1}}{\tilde{t}_{i,p+i}} + \frac{1}{t^{(i-1)}_{m,n}} + \tilde{t}_{i,p+i} \left(\frac{1}{t_{i+1,p+i}} + \frac{1}{t_{i,p+i+1}}\right) - \frac{1}{t^{(i)}_{m,n}} \right] + \frac{1}{t^{(m-1)}_{m,n}} \, .
\end{split}
\]
Noticing that
\[
\sum_{i=1}^{m-1} \left[ \frac{1}{t^{(i-1)}_{m,n}} - \frac{1}{t^{(i)}_{m,n}} \right] + \frac{1}{t^{(m-1)}_{m,n}}
= \frac{1}{w_{m,n}} \, ,
\]
as the summation is telescopic and $t^{(0)}_{m,n} = \tilde{t}_{m,n} = w_{m,n}$, we conclude that~\eqref{eq:sumInvWeights_induction2} holds.
\end{proof}

\begin{remark}
All the geometric correspondences of Section~\ref{sec:geometricCorrespondences} have been defined via composition of local maps and hence, essentially, via a recursive procedure.
Furthermore, the geometric RSK correspondence possesses another recursive structure on the border output entries, as we now explain.
Let $(m,n)$ be a border box of a partition $\lambda$, $\bm{w}\in\R_{>0}^{\lambda}$, $\bm{s} := \sfK(\bm{w})$, and $\bm{t}:= \sfB(\bm{w})$.
It is immediate to see from the definition~\eqref{eq:gRSK} of geometric RSK correspondence that $s_{m,n} = (s_{m-1,n} + s_{m,n-1})w_{m,n}$.
However, the geometric Burge correspondence lacks such an obvious recursive structure, in the sense that $t_{m,n}$ cannot be written as a function of the entries of $\bm{t}$ and $\bm{w}$ in the boxes neighboring $(m,n)$.
This is also reflected by the fact that the right-hand side of~\eqref{eq:non-intersPaths_partitionFn}, which expresses $t_{m,n}$ in terms of directed paths on the input array $\bm{w}$, also lacks a recursive structure if viewed as a function of $(m,n)$.
Therefore, it is not clear how to prove~\eqref{eq:non-intersPaths_partitionFn} inductively, i.e.\ similarly to the proof strategy of Proposition~\ref{prop:sumInvWeights}.
On the other hand, it is easy to see that both sides of~\eqref{eq:non-intersPaths_diagProd} do have an inductive structure, so formula~\eqref{eq:non-intersPaths_diagProd} could be also proven inductively from the definition of geometric Burge correspondence.

\end{remark}

We now state and prove the volume-preserving property for the geometric Burge correspondence in log-log variables.

\begin{theorem}
\label{thm:volume}
Let $\bm{w}\in \R_{>0}^{\lambda}$ and $\bm{t}:= \sfB(\bm{w})$.
Then, the map
\[
(\log w_{i,j})_{(i,j)\in \lambda} \mapsto (\log t_{i,j})_{(i,j)\in \lambda}
\]
has Jacobian $\pm 1$.
\end{theorem}

\begin{proof}
As $\sfB$ can be written as a composition of $c_{k,l}$'s and $d^{k,l}_{i,j}$'s (see~\eqref{eq:gBurge} and~\eqref{eq:tau}), it suffices to show that both these types of local maps are volume preserving in log-log variables.
This property is immediate for $c_{k,l}$, so we will prove it for $d^{k,l}_{i,j}$ only.
We will also suppose that $i,j>1$, as the proof simplifies when $i=1$ or $j=1$.
Set
\begin{align*}
x_1 &:= \log w_{i,j} \, ,
&x_3 &:= \log w_{i-1,j} \, ,
&x_5 &:= \log w_{i+1,j} \, , \\
x_2 &:= \log w_{k,l} \, ,
&x_4 &:= \log w_{i,j-1} \, ,
&x_6 &:= \log w_{i,j+1} \, .
\end{align*}
Looking at the definition~\eqref{eq:d} of $d^{k,l}_{i,j}$, we define the transformation $F\colon \R^6 \to \R^6$ with $i$-th component
\begin{align*}
F_i(x_1,\dots,x_6) =
\begin{cases}
-\log \left[ \e^{-x_1} + \e^{-x_2}(\e^{x_3} + \e^{x_4})^{-1} \right] &i=1 \, , \\
\log\left[ \e^{x_2 - 2x_1} (\e^{x_3} + \e^{x_4}) + \e^{-x_1} \right] - \log\left[\e^{-x_5} + \e^{-x_6}\right] &i=2 \, , \\
x_i &3\leq i\leq 6 \, .
\end{cases}
\end{align*}
To obtain the Jacobian of $F$, it suffices to compute the partial derivatives of $F_1$ and $F_2$ with respect of $x_1$ and $x_2$.
Setting
\[
g(x_1,x_2,x_3,x_4) := \frac{\e^{-x_1}}{ \e^{-x_1} + \e^{-x_2}(\e^{x_3} + \e^{x_4})^{-1} } \, ,
\]
one can easily obtain:
\[
\frac{\partial F_1}{\partial x_1} = g \, , \qquad\quad
\frac{\partial F_1}{\partial x_2} = 1-g \, , \qquad\quad
\frac{\partial F_2}{\partial x_1} = -1-g \, , \qquad\quad
\frac{\partial F_2}{\partial x_2} = g \, .
\]
Therefore, the modulus of the Jacobian of $F$ is given by
\[
\abs{\frac{\partial F_1}{\partial x_1} \frac{\partial F_2}{\partial x_2} - \frac{\partial F_1}{\partial x_2} \frac{\partial F_2}{\partial x_1}}
= \abs{g^2 - (1-g)(-1-g)} = 1 \, ,
\]
as desired.
\end{proof}

\section{The geometric Burge correspondence on symmetric arrays}
\label{sec:Burge_symmetric}

A \textbf{self-conjugate} partition, or equivalently a symmetric Young diagram, is a partition $\lambda$ such that $(i,j)\in\lambda$ if and only if $(j,i)\in \lambda$.
If $\lambda$ is a self-conjugate partition, an array $\bm{w}$ of shape $\lambda$ is called \textbf{symmetric} if $\bm{w}^{\sfT} = \bm{w}$, i.e.\ $w_{i,j}=w_{j,i}$ for all $(i,j)\in \lambda$.

The aim of this section is to explain how the geometric Burge correspondence behaves when restricted to symmetric arrays.
This will be a key ingredient for the polymer analysis in Section~\ref{sec:polymer}.

\begin{proposition}
\label{prop:Burge_symmetric}
For any $\bm{w}\in\R_{>0}^{\lambda}$, we have that $\sfB(\bm{w}^{\sfT}) = \sfB(\bm{w})^{\sfT}$.
In particular, if $\bm{w}$ is a symmetric array, then $\sfB(\bm{w})$ also is.
\end{proposition}
\begin{proof}
The equality $\sfB(\bm{w}^{\sfT}) = \sfB(\bm{w})^{\sfT}$ follows from the fact that $\sfB$ is a composition of local maps $c_{k,l}$'s and $d^{k,l}_{i,j}$'s, which trivially commute with the transposition map (see definitions~\eqref{eq:c} and~\eqref{eq:d}).
In particular, if $\bm{w}$ is symmetric, i.e.\ $\bm{w}=\bm{w}^{\sfT}$, then $\sfB(\bm{w}) = \sfB(\bm{w})^{\sfT}$, which means that $\sfB(\bm{w})$ is also symmetric.
\end{proof}

The latter proposition implies that the geometric Burge correspondence on symmetric arrays of shape $\lambda$ can be restricted to a bijection on arrays indexed by the ``upper part'' of $\lambda$, i.e. $\lambda^{\up}:= \{(i,j)\in \lambda\colon i\leq j\}$ (notice that, in general, $\lambda^{\up}$ is \emph{not} a partition).
Namely, there exists a bijection
\[
\sfB^{\up}\colon \R_{>0}^{\lambda^{\up}} \to \R_{>0}^{\lambda^{\up}} \, ,
\qquad\quad
(w_{i,j})_{(i,j)\in\lambda^{\up}} \mapsto (t_{i,j})_{(i,j)\in\lambda^{\up}}
\]
such that $\sfB(\bm{w})|_{\lambda^{\up}} = \sfB^{\up}(\bm{w}|_{\lambda^{\up}})$ for all $\bm{w}\in \R_{>0}^{\lambda}$.
One obvious way to obtain the output of $\sfB^{\up}$ is to take the input array indexed by $\lambda^{\up}$, symmetrize it about the diagonal, apply the geometric Burge correspondence, thus obtaining (via Proposition~\ref{prop:Burge_symmetric}) another symmetric array, which can be restricted back to $\lambda^{\up}$.
Another equivalent way is to define new local maps, by slightly modifying the original definitions, and apply them directly to the restricted array $(w_{i,j})_{(i,j)\in\lambda^{\up}}$.
More precisely, the new local maps $\sfc^{\up}_{k,l}$ and $\sfd^{k,l,\up}_{i,j}$ need to be defined as follows:
\begin{itemize}
\item If $i<j$ and $k<l$, then $\sfc^{\up}_{i,j} := \sfc_{i,j}$ and $\sfd^{k,l,\up}_{i,j} := \sfd^{k,l}_{i,j}$.
\item If $i=j$ and $k=l$, then $\sfc^{\up}_{k,l}$ and $\sfd^{k,l,\up}_{i,j}$ only modify the following entries:
\begin{align}
\label{eq:c_symm}
\sfc_{i,i}^{\up} &\colon w_{i,i} \longmapsto 2 w_{i-1,i} w_{i,i} \, , \\
\label{eq:d_symm}
\sfd^{k,k,\up}_{i,i} &\colon
\begin{cases}
w_{i,i} \longmapsto \left( \dfrac{1}{w_{i,i}} + \dfrac{1}{2 w_{i-1,i} w_{k,k} } \right)^{-1} \, , \\
w_{k,k} \longmapsto \left( \dfrac{2 w_{i-1,i} w_{k,k}}{w_{i,i}^2} + \dfrac{1}{w_{i,i}} \right) \dfrac{w_{i,i+1}}{2} \, ,
\end{cases}
\end{align}
with the usual conventions that $w_{0,1}=1/2$ and $w_{0,k}=0$ for all $k>1$.
\end{itemize}
These new local maps can be obtained by just specializing~\eqref{eq:c}-\eqref{eq:d} to the symmetric case.
We then define $\tau^{\up}_{k,l}$, for all $k\leq l$, by just replacing $\sfc_{k,l}$ with $\sfc^{\up}_{k,l}$ and each $\sfd^{k,l}_{i,j}$ with $\sfd^{k,l,\up}_{i,j}$ in the definition~\eqref{eq:tau} of $\tau_{k,l}$.
We can finally construct $\sfB^{\up}$ by setting
\begin{equation}
\label{eq:gBurge_symm}
\sfB^{\up} := \tau_{i_n,j_n}^{\up} \circ \tau_{i_{n-1},j_{n-1}}^{\up} \circ \dots \circ \tau_{i_1,j_1}^{\up} \, ,
\end{equation}
where $((i_1,j_1),\dots,(i_n,j_n))$ is any sequence of distinct boxes such that, for each $1\leq k\leq n$, 
\[
\lambda^{(k)} := \{(i_1,j_1)\dots,(i_k,j_k)\}_{k=1}^n \cup \{(j_1,i_1)\dots,(j_k,i_k)\}_{k=1}^n
\]
is a Young diagram, and $\lambda^{(n)} = \lambda$.

The properties stated in Propositions~\ref{prop:non-intersPaths} and~\ref{prop:sumInvWeights} automatically hold when the geometric Burge correspondence is restricted to symmetric arrays.
On the other hand, the volume-preserving property does not follow immediately and is addressed in the next theorem.

\begin{theorem}
\label{thm:volume_sym}
Let $\bm{w}\in \R^{\lambda}_{>0}$ be a symmetric array and let $\bm{t} := \sfB(\bm{w})$.
Then, the map
\[
(\log w_{i,j})_{(i,j)\in \lambda, \, i\leq j} \mapsto (\log t_{i,j})_{(i,j)\in \lambda, \, i\leq j}
\]
has Jacobian $\pm 1$.
\end{theorem}

\begin{proof}
Set $\lambda^{\up}:= \{(i,j)\in \lambda\colon i\leq j\}$.
As explained above, we have that $\bm{t}|_{\lambda^{\up}} = \sfB^{\up}(\bm{w}|_{\lambda^{\up}})$.
Since $\sfB^{\up}$ can be defined by~\eqref{eq:gBurge_symm}, it is a composition of $\sfc^{\up}_{i,j}$'s and $\sfd^{k,l,\up}_{i,j}$'s.
Therefore, it suffices to prove that these modified local maps are volume preserving in log-log variables.
If $i<j$ and $k<l$, the new local maps coincide with the old ones, which possess this property as shown in the proof of Theorem~\ref{thm:volume}.
If $i=j$ and $k=l$, the new local maps are given by~\eqref{eq:c_symm} and~\eqref{eq:d_symm}: in this case, the check is still totally analogous to the one done for Theorem~\ref{thm:volume}, so we omit it.
\end{proof}

\section{Polymer replicas and Whittaker functions}
\label{sec:polymer}

In this section we study a polymer model in a persymmetric environment, as discussed in Subsection~\ref{subsec:intro_contribution}.
In particular, thanks to the direct connection between the point-to-point partition function in a persymmetric environment and the replica partition function (see \eqref{eq:overlappingPolymers}-\eqref{eq:replica}), we determine the distribution of the latter as a Whittaker measure.

Let us first introduce Whittaker functions.
For any triangular array $\bm{z} = (z_{i,j})_{1\leq j\leq i\leq n}$, define the \textbf{energy} of $\bm{z}$ to be the functional
\[
\mathcal{E}(\bm{z})
:= \sum_{i=1}^{n-1} \sum_{j=1}^i \left(\frac{z_{i+1,j+1}}{z_{i,j}} + \frac{z_{i,j}}{z_{i+1,j}} \right) \, .
\]
Define also the \textbf{(geometric) type} of $\bm{z}$ to be the $n$-vector, denoted by $\type(\bm{z})$, whose $i$-th component is the ratio between the product of the $i$-th row of $\bm{z}$ and the product of its $(i-1)$-th row:
\begin{equation}
\label{eq:type}
\type(\bm{z})_i := \frac{\prod_{j=1}^i z_{i,j}}{\prod_{j=1}^{i-1} z_{i-1,j}} \, , \qquad\qquad 1\leq i\leq n \, .
\end{equation}
We then define the $\GL_n(\R)$-\textbf{Whittaker function} with parameter $\bm{\alpha}=(\alpha_1,\dots,\alpha_n)\in\C^n$ and argument $\bm{x}=(x_1,\dots,x_n)\in\R_{>0}^n$ as
\begin{equation}
\label{eq:Whittaker}
\Psi^{n}_{\bm{\alpha}}(\bm{x})
:= \int_{\mathcal{T}^n(\bm{x})} \prod_{i=1}^n \type(\bm{z})_i^{\alpha_i} \cdot
\e^{-\mathcal{E}(\bm{z})} \prod_{\substack{1\leq i <n \\ 1\leq j\leq i}} \frac{\diff z_{i,j}}{z_{i,j}} \, ,
\end{equation}
where $\mathcal{T}^n(\bm{x})$ is the set of all triangular arrays $\bm{z} = (z_{i,j})_{1\leq j\leq i\leq n}$ with positive entries and bottom row $(z_{n,1},\dots,z_{n,n}) = (x_1,\dots,x_n) = \bm{x}$.

Recall that a random variable $Y$ follows an inverse gamma distribution with parameters $\alpha>0$ and $\beta>0$ if
\begin{equation}
\label{eq:inverseGamma}
\P(Y\in \diff y) 
= \frac{\beta^{\alpha}}{\Gamma(\alpha)} y^{-\alpha} \exp\left(-\frac{\beta}{y}\right) \1_{\{y>0\}} \frac{\diff y}{y} \, ,
\end{equation}
in which case we write $Y \sim \invGamma(\alpha,\beta)$.

Fix now parameters $\bm{\alpha} = (\alpha_1,\dots, \alpha_n)\in\R_{>0}^n$ and $\beta\in\R_{>0}$.
Consider a random $n\times n$ symmetric matrix $\bm{W} = (W_{i,j})_{1\leq i,j\leq n}$ with entries $(W_{i,j})_{i\leq j}$ independent and inverse gamma distributed as follows:
\[
W_{i,j} \sim
\begin{cases}
\invGamma(\alpha_i, \beta)
&\text{if } 1\leq i=j\leq n \, , \\
\invGamma(\alpha_i + \alpha_j,1)
&\text{if } 1\leq i<j\leq n \, .
\end{cases}
\]
Namely, the joint distribution of the upper triangular entries of $\bm{W}$ is
\begin{equation}
\label{eq:inverseGamma_symm}
\begin{split}
&\P(W_{i,j} \in \diff w_{i,j}\colon i\leq j) \\
= \, &\frac{1}{c_{\bm{\alpha},\beta}}
\left[ \prod_{i=1}^n w_{i,i}^{-\alpha_i} \prod_{i<j} w_{i,j}^{-\alpha_i - \alpha_j} \right]
\exp\left\{ - \sum_{i=1}^n \frac{\beta}{w_{i,i}} -\sum_{i<j} \frac{1}{w_{i,j}} \right\}
\prod_{i\leq j} \1_{\{w_{i,j}>0\}} \frac{\diff w_{i,j}}{w_{i,j}} \, ,
\end{split}
\end{equation}
with normalization constant
\begin{equation}
\label{eq:normalization}
c_{\bm{\alpha},\beta}
:= \beta^{-\sum_{i=1}^n \alpha_i}
\prod_{i=1}^n \Gamma(\alpha_i)
\prod_{i<j} \Gamma(\alpha_i + \alpha_j) \, .
\end{equation}

The main theorem of this section states that the diagonal entries of the image of $\bm{W}$ under the geometric Burge correspondence have a joint density proportional to a $\GL_n(\R)$-Whittaker function (with an exponential prefactor).

\begin{theorem}
\label{thm:WhittakerMeasure}
If $\bm{W}$ is an $n\times n$ symmetric matrix distributed according to~\eqref{eq:inverseGamma_symm} and $\bm{T} := \sfB(\bm{W})$, then
\begin{equation}
\label{eq:WhittakerMeasure}
\P(T_{i,i} \in \diff x_i \colon 1\leq i\leq n)
= \frac{1}{c_{\bm{\alpha},\beta}} 
\e^{-\beta/x_n}
\Psi^n_{-\bm{\alpha}}(x_1,\dots,x_n) \prod_{i=1}^n \1_{\{x_i>0\}} \frac{\diff x_i}{x_i} \, ,
\end{equation}
where the constant $c_{\bm{\alpha},\beta}$ is defined by~\eqref{eq:normalization}.
\end{theorem}
\begin{proof}
Our strategy consists in computing the push-forward measure that the distribution of $\bm{W}$ induces on $\bm{T}$, using the properties of the geometric Burge correspondence obtained in Sections~\ref{sec:Burge_properties} and~\ref{sec:Burge_symmetric}.

Let $\bm{w}\in\R_{>0}^{n\times n}$ be a symmetric matrix; then, $\bm{t}:= \sfB(\bm{w})$ is also symmetric by Proposition~\ref{prop:Burge_symmetric}.
Moreover, by~\eqref{eq:non-intersPaths_diagProd}, we have that
\[
\prod_{j=1}^n w_{i,j}
= \frac{P_{n-i}(\bm{t})}{P_{n-i+1}(\bm{t})} \, ,
\]
where $P_k(\bm{t})$ is the product of the $k$-th diagonal of $\bm{t}$, as defined in~\eqref{eq:diagProd}.
It follows that
\[
\begin{split}
\prod_{i=1}^n w_{i,i}^{-\alpha_i} \prod_{i<j} w_{i,j}^{-\alpha_i - \alpha_j}
= \prod_{i,j=1}^n w_{i,j}^{-\alpha_i}
= \prod_{i=1}^n \left(\prod_{j=1}^n w_{i,j} \right)^{-\alpha_i}
= \prod_{i=1}^n \left( \frac{P_{n-i}(\bm{t})}{P_{n-i+1}(\bm{t})} \right)^{-\alpha_i}  \, .
\end{split}
\]

On the other hand, Proposition~\ref{prop:sumInvWeights} implies that
\begin{align*}
\sum_{i=1}^n \frac{\beta}{w_{i,i}}
&= \frac{\beta}{t_{1,1}} \, , \\
\sum_{i<j} \frac{1}{w_{i,j}}
&= \frac{1}{2} \sum_{i\neq j} \frac{1}{w_{i,j}}
= \frac{1}{2} \sum_{\substack{1\leq i,j\leq n \\ (i,j)\neq (1,1)}} \frac{t_{i-1,j} + t_{i,j-1}}{t_{i,j}}
= \sum_{1<i\leq j\leq n} \frac{t_{i-1,j}}{t_{i,j}} + \sum_{1\leq i\leq j<n} \frac{t_{i,j-1}}{t_{i,j}} \, .
\end{align*}
By Theorem~\ref{thm:volume_sym} the map $(\log w_{i,j})_{i\leq j} \mapsto (\log t_{i,j})_{i\leq j}$ has Jacobian $\pm 1$, hence the push-forward of~\eqref{eq:inverseGamma_symm} is
\[
\begin{split}
&\P(T_{i,j} \in \diff t_{i,j} \colon 1\leq i\leq j\leq n)
= \frac{1}{c_{\bm{\alpha},\beta}}
\prod_{i=1}^n \left( \frac{P_{n-i}(\bm{t})}{P_{n-i+1}(\bm{t})} \right)^{-\alpha_i} \\
&\qquad\qquad\times \exp\left\{ - \left[ \frac{\beta}{t_{1,1}} + \sum_{1<i\leq j\leq n} \frac{t_{i-1,j}}{t_{i,j}} + \sum_{1\leq i\leq j<n} \frac{t_{i,j-1}}{t_{i,j}} \right] \right\} 
\prod_{i\leq j} \1_{\{t_{i,j}>0\}} \frac{\diff t_{i,j}}{t_{i,j}} \, .
\end{split}
\]
To obtain the joint density of $(T_{1,1},\dots,T_{n,n})$, one has to integrate out all $t_{i,j}$'s with $i<j$ in the latter expression.
If we then reindex the variables by setting $t_{i,j} = z_{n-j+i,n-j+1}$ for all $1\leq i\leq j\leq n$, we obtain the right-hand side of~\eqref{eq:WhittakerMeasure}, where the Whittaker function $\Psi^n_{-\bm{\alpha}}$ is defined by~\eqref{eq:Whittaker}.
\end{proof}

Since the right-hand side of~\eqref{eq:WhittakerMeasure} is a probability density, we obtain as a consequence an explicit integral formula for Whittaker functions in terms of gamma functions.
\begin{corollary}
\label{eq:cauchy}
For $\bm{\alpha}\in\R_{>0}^n$ and $\beta\in\R_{>0}$, we have that
\[
\int_{\R_{>0}^n} \e^{-\beta/x_n} \Psi^n_{-\bm{\alpha}}(x_1,\dots,x_n) \prod_{i=1}^n \frac{\diff x_i}{x_i}
= \beta^{-\sum_{i=1}^n \alpha_i} \prod_{i=1}^n \Gamma(\alpha_i)
\prod_{1\leq i<j\leq n} \Gamma(\alpha_i+\alpha_j) \, .
\]
\end{corollary}
The latter can be seen as the analog of a Cauchy-Littlewood identity in our setting.
It is equivalent to an integral identity for the Mellin transform of a Whittaker function, conjectured by Bump and Friedberg~\cite{bumpFriedberg90} and proved by Stade~\cite{stade01} -- see~\cite[Section~7]{oConnellSeppalainenZygouras14}.

Let us now link Theorem~\ref{thm:WhittakerMeasure} to the polymer models introduced in Section~\ref{sec:intro}.
Consider again the symmetric log-gamma random environment $\bm{W}$ of~\eqref{eq:inverseGamma_symm}.
For $1\leq k\leq n$, define
\begin{equation}
\label{eq:polymerPartitionFn*}
Z^{*(k)}_{n,n} :=
\sum_{(\pi_1,\dots,\pi_k) \in \Pi^{*(k)}_{n,n}} \prod_{(i,j)\in \pi_1 \cup \cdots \cup \pi_k} W_{i,j} \, ,
\end{equation}
where $\Pi^{*(k)}_{n,n}$ is the set of $k$-tuples of non-intersecting directed lattice paths in $\N^2$ starting at $(n,1)$, $(n,2)$, \dots, $(n,k)$ and ending at $(1,n-k+1)$, $(1,n-k+2)$, \dots, $(1,n)$ respectively.
In particular, $Z^*_{n,n} = Z^{*(1)}_{n,n}$ is the \textbf{dual point-to-point polymer partition function} in a symmetric environment.
By~\eqref{eq:non-intersPaths}, each $Z^{*(k)}_{n,n}$ can be expressed as the product of some diagonal entries of the image $\bm{T}$ of $\bm{W}$ under the geometric Burge correspondence:
\[
Z^{*(k)}_{n,n} = T_{1,1} \cdots T_{k,k} \, , \qquad\qquad
1\leq k\leq n \, .
\]
Thus, the right-hand side of~\eqref{eq:WhittakerMeasure} is precisely the joint density of the random vector
\begin{equation}
\label{eq:dualPartition}
\left(T_{1,1},T_{2,2},\dots,T_{n,n}\right)
= \left(Z^{*(1)}_{n,n}, \frac{Z^{*(2)}_{n,n}}{Z^{*(1)}_{n,n}}, \dots, \frac{Z^{*(n)}_{n,n}}{Z^{*(n-1)}_{n,n}}\right) \, .
\end{equation}

Obviously, the dual point-to-point polymer partition function in the \emph{symmetric} environment $\bm{W}$ can be alternatively viewed as the point-to-point polymer partition function in the \emph{persymmetric} environment obtained by reversing the columns (or the rows) of $\bm{W}$.
Furthermore, the latter can also be reinterpreted as a \textbf{replica partition function} of two polymer paths constrained to coincide at the 
endpoint -- see explanations in Section~\ref{sec:intro}, around~\eqref{eq:overlappingPolymers}-\eqref{eq:replica}.
In particular, let $\Zrepl{n}$ be the replica partition function~\eqref{eq:replica} on the modified log-gamma environment $(W'_{i,j})_{i+j\leq n+1}$ given by
\begin{equation}
\label{eq:inverseGamma_replica}
\begin{split}
&\P(W_{i,j}' \in \diff w_{i,j}\colon i+ j\leq n+1)
= \frac{2^n}{c_{\bm{\alpha},\beta}}
\left[ \prod_{i=1}^n w_{i,n-i+1}^{-2 \alpha_{i}} \prod_{i+j\leq n} w_{i,j}^{-\alpha_i - \alpha_{n-j+1}} \right] \\
&\qquad\qquad\qquad \times \exp\left\{ - \sum_{i=1}^n \frac{\beta}{w^2_{i,n-i+1}} -\sum_{i+j\leq n} \frac{1}{w_{i,j}} \right\}
\prod_{i+j\leq n+1} \1_{\{w_{i,j}>0\}} \frac{\diff w_{i,j}}{w_{i,j}} \, ,
\end{split}
\end{equation}
with normalization constant $c_{\bm{\alpha},\beta}$  as in~\eqref{eq:normalization}.
It then follows that
\[
\Zrepl{n} \text{ is equal in distribution to } Z^*_{n,n} \, .
\]

Thanks to Theorem~\ref{thm:WhittakerMeasure} and~\eqref{eq:dualPartition}, we
thus arrive at the following theorem:
\begin{theorem}
The Laplace transform of the replica partition function $\Zrepl{n}$ defined in~\eqref{eq:replica} on the log-gamma environment~\eqref{eq:inverseGamma_replica} is given by the integral formula
\begin{equation}
\label{eq:polymerLaplace}
\E\left[ \e^{-r \Zrepl{n}} \right]
= \frac{1}{c_{\bm{\alpha},\beta}} \int_{\R_{>0}^n} \e^{-rx_1 - \beta/x_n} \, \Psi^n_{-\bm{\alpha}}(x_1,\dots,x_n) \prod_{i=1}^n \frac{\diff x_i}{x_i} \, .
\end{equation}
\end{theorem}

We close by discussing a non-trivial identity in distribution, in the context of symmetryzed polymers, between the dual and the usual partition functions, which have been analyzed in the present work and in~\cite{oConnellSeppalainenZygouras14} respectively.
Let us consider a symmetric matrix $\bm{W}$ distributed as in~\eqref{eq:inverseGamma_symm} with $\beta=1/2$.
Define, for $1\leq k\leq n$,
\begin{equation}
\label{eq:polymerPartitionFn}
Z^{(k)}_{n,n} :=
\sum_{(\pi_1,\dots,\pi_k) \in \Pi^{(k)}_{n,n}} \prod_{(i,j)\in \pi_1 \cup \cdots \cup \pi_k} W_{i,j} \, ,
\end{equation}
where $\Pi^{(k)}_{n,n}$ is the set of $k$-tuples of non-intersecting directed lattice paths in $\N^2$ starting at $(1,1)$, $(1,2)$, \dots, $(1,k)$ and ending at $(n,n-k+1)$, $(n,n-k+2)$, \dots, $(n,n)$ respectively.
In particular, $Z_{n,n} = Z^{(1)}_{n,n}$ is the usual partition function from $(1,1)$ to $(n,n)$ on $\bm{W}$.
It was shown in~\cite[Section~5]{oConnellSeppalainenZygouras14} that the vector
\[
\left(Z^{(1)}_{n,n}, \frac{Z^{(2)}_{n,n}}{Z^{(1)}_{n,n}}, \dots, \frac{Z^{(n)}_{n,n}}{Z^{(n-1)}_{n,n}}\right)
\]
has exactly the density given by the right-hand side of~\eqref{eq:WhittakerMeasure} for $\beta=1/2$.
In particular, for this specific symmetric environment, $Z_{n,n}$ and $Z^*_{n,n}$ turn out to be identically distributed.

When $n=2$, this reduces to the following identity in law: if $X,Y$ and $Z$ are independent inverse gamma random variables with respective parameters $a,b$ and $a+b$, then the random variables $(X+Y)Z^2$ and $XYZ$ have the same law.
This can be seen as a consequence of Lukacs' theorem~\cite[\S~1]{lukacs55}, as follows.
Let us write $U:=X^{-1}+Y^{-1}$ and $V:=X^{-1} Y$.
Since $U^{-1}$ and $Z$ are independent and both inverse gamma distributed with parameter $a+b$, we have that $U^{-2} Z$ and $U^{-1} Z^2$ are equally distributed.
Moreover, by Lukacs' theorem, $U$ and $V$ are independent.
It follows that $XYZ = U^{-2} Z (1+V)^2 V^{-1}$ has the same law as $(X+Y)Z^2 = U^{-1} Z^2 (1+V)^2 V^{-1}$, as required.

\appendix

\section{Proof of Proposition~\ref{prop:compDiagMaps}}
\label{app:proofProp}

We will prove
\begin{equation}
\label{eq:compDiagMaps2}
\sigma_{p,q} \rho_{p,q+1} \tau_{p,q}
= \tau_{p,q+1} \rho_{p,q} \sigma_{p-1,q} \sfe_{p,q}^{p,q+1} \, ,
\end{equation}
as an identity of maps acting on an array $\bm{w} = (w_{i,j})_{(i,j)\in\ \lambda} \in \R_{>0}^{\lambda}$, where $\lambda$ is a partition such that $(p,q),(p,q+1),(p-1,q)\in \lambda$.
We will proceed by induction on the number of entries of $\bm{w}$ that the maps appearing in~\eqref{eq:compDiagMaps2} modify, i.e.\
\begin{equation}
\label{eq:nEntries2Diag}
n_{p,q}
:= p\wedge q + p\wedge (q+1)
= \begin{cases}
2p &\text{if } p\leq q \, , \\
2q+1 &\text{if } p\geq q+1 \, .
\end{cases}
\end{equation}
Note that $n_{p,q}\geq 3$, as $(p,q),(p,q+1),(p-1,q)\in \lambda$.

When $n_{p,q} =3$, i.e.\ $p=2,q=1$, the only entries that the maps in~\eqref{eq:compDiagMaps2} use and/or modify are $w_{i,j}$ with $1\leq i,j\leq 2$.
We may then assume that $\lambda=(2,2)$ and compute:
\[
\begin{split}
\begin{matrix}
w_{1,1} &w_{1,2} \\
w_{2,1} &w_{2,2}
\end{matrix}
\quad&\xmapsto{\tau_{2,1}}\quad
\begin{matrix}
w_{1,1} &w_{1,2} \\
w_{1,1} w_{2,1} &w_{2,2}
\end{matrix} \\
&\xmapsto{\rho_{2,2}}\quad
\begin{matrix}
w_{1,2}w_{2,1}(w_{1,2}+w_{1,1}w_{2,1})^{-1} &w_{1,2} \\
w_{1,1} w_{2,1} &w_{2,2}(w_{1,2}+w_{1,1}w_{2,1})
\end{matrix} \\
&\xmapsto{\sigma_{2,1}}\quad
\begin{matrix}
w_{1,2}w_{2,1}(w_{1,2}+w_{1,1}w_{2,1})^{-1} &w_{1,2} \\
w_{1,2} w_{2,2} w_{1,1}^{-1} &w_{2,2}(w_{1,2}+w_{1,1}w_{2,1})
\end{matrix}
\end{split}
\]
and, on the other hand,
\[
\begin{split}
\begin{matrix}
w_{1,1} &w_{1,2} \\
w_{2,1} &w_{2,2}
\end{matrix}
\quad&\xmapsto{\sfe_{2,1}^{2,2}}\quad
\begin{matrix}
w_{1,1} &w_{1,2} \\
w_{2,2} &w_{2,1}
\end{matrix} \\
&\xmapsto{\sigma_{1,1}}\quad
\begin{matrix}
w_{1,2}w_{1,1}^{-1} &w_{1,2} \\
w_{2,2} &w_{2,1}
\end{matrix} \\
&\xmapsto{\rho_{2,1}}\quad
\begin{matrix}
w_{1,2}w_{1,1}^{-1} &w_{1,2} \\
w_{1,2} w_{2,2} w_{1,1}^{-1} &w_{2,1}
\end{matrix} \\
&\xmapsto{\tau_{2,2}}\quad
\begin{matrix}
w_{1,2}w_{2,1}(w_{1,2}+w_{1,1}w_{2,1})^{-1} &w_{1,2} \\
w_{1,2} w_{2,2} w_{1,1}^{-1} &w_{2,2}(w_{1,2}+w_{1,1}w_{2,1})
\end{matrix} \, .
\end{split}
\]
The output arrays coincide, so~\eqref{eq:compDiagMaps2} holds for $n_{p,q}=3$.

The other ``simple'' cases $n_{p,q}=4,5,6$ are also dealt by hand (we omit the straightforward but tedious checks), as the inductive procedure that we use in the following works when $n_{p,q} \geq 7$.
Given $n\geq 6$, we then assume by induction that~\eqref{eq:compDiagMaps2} holds when $n_{p,q} = n$ and aim to deduce that it also holds when $n_{p,q}=n+1$.
To fix the ideas, we will prove this when $n$ is even -- the proof when $n$ is odd is analogous.
If $n$ is even then $n_{p,q}=n+1$ is odd and, by~\eqref{eq:nEntries2Diag}, we have that $n_{p,q}=2q+1$ and $p\geq q+1$.
Set $m:= p-q \geq 1$.
Notice that $n_{q,q}=2q=n$, so by the induction hypothesis we have that
\begin{equation}
\label{eq:compDiagMaps3}
\sigma_{q,q} \rho_{q,q+1} \tau_{q,q}
= \tau_{q,q+1} \rho_{q,q} \sigma_{q-1,q} \sfe_{q,q}^{q,q+1} \, .
\end{equation}
Keeping in mind that we eventually wish to deduce~\eqref{eq:compDiagMaps2}, we now apply~\eqref{eq:compDiagMaps3} to the array obtained by just ignoring the first $m$ rows of $\bm{w}$.
This operation, viewed on the whole array $\bm{w}$, corresponds to ``shifting downwards'' the diagonal maps of~\eqref{eq:compDiagMaps3} by $m$ rows.
The local maps involved in the diagonal maps of~\eqref{eq:compDiagMaps3} with subscript $(i,j)$, $i\geq 2$, then simply become the corresponding maps shifted downwards by $m$ rows: e.g., $\sfd_{i,j}^{k,l}$ just becomes $\sfd_{m+i,j}^{m+k,l}$ for all $i\geq 2$.
However, when shifted downwards by $m$ rows, $\sfa_{1,2}$ does \emph{not} become $\sfa_{m+1,2}$.
The reason is that the two maps act differently with respect to the previous row: by definition, $\sfa_{1,2}$ does not use any entries in row $0$ (which does not even exist), whereas $\sfa_{m+1,2}$ does use entries in row $m$ (in particular, it uses $w_{m,2}$).
The same holds for all the other local maps with subscript on the first row: $\sfa_{1,j}$ and $\sfd^{k,l}_{1,j}$, after the downwards shift, do \emph{not} become $\sfa_{m+1,j}$ and $\sfd^{m+k,l}_{m+1,j}$ respectively.
Instead of the latter, we will need to use:
\begin{align*}
&\tilde{\sfa}_{m+1,j}\colon w_{m+1,j} \longmapsto \frac{w_{m+1,j-1}}{w_{m+1,j}} \left(\frac{1}{w_{m+2,j}} + \frac{1}{w_{m+1,j+1}} \right)^{-1} \, , \\
&\tilde{\sfd}^{m+k,l}_{m+1,j}\colon
\begin{cases}
w_{m+1,j} \longmapsto \left( \dfrac{1}{w_{m+1,j}} + \dfrac{1}{w_{m+k,l} w_{m+1,j-1}} \right)^{-1} \, , \\
w_{m+k,l} \longmapsto \left( \dfrac{w_{m+k,l} w_{m+1,j-1}}{w_{m+1,j}^2} + \dfrac{1}{w_{m+1,j}} \right) \left(\dfrac{1}{w_{m+2,j}} + \dfrac{1}{w_{m+1,j+1}} \right)^{-1} \, ,
\end{cases}
\end{align*}
with the convention (different from the usual one!) that $w_{m+1,0}=1$.
We thus conclude that~\eqref{eq:compDiagMaps3}, decomposed into local maps and downwards shifted (so that it really acts only on rows $i\geq m+1$ of $\bm{w}$), reads as
\begin{equation}
\label{eq:compDiagMapsShift}
\begin{split}
&[\tilde{\sfa}_{m+1,1} \sfa_{m+2,2} \cdots \sfa_{m+q-1,q-1} \sfb_{m+q,q}] \circ
[\tilde{\sfa}_{m+1,2} \sfa_{m+2,3} \cdots \sfa_{m+q-1,q} \sfc_{m+q,q+1}] \\
&\qquad\qquad\qquad\qquad\qquad\qquad
\circ [\sfc_{m+q,q} \sfd^{m+q,q}_{m+q-1,q-1} \cdots \sfd^{m+q,q}_{m+2,2} \tilde{\sfd}^{m+q,q}_{m+1,1}] \\
= \, &[\sfc_{m+q,q+1} \sfd^{m+q,q+1}_{m+q-1,q} \cdots \sfd^{m+q,q+1}_{m+2,3} \tilde{\sfd}^{m+q,q+1}_{m+1,2}] \circ
[\tilde{\sfa}_{m+1,1} \sfa_{m+2,2} \cdots \sfa_{m+q-1,q-1} \sfc_{m+q,q}] \\
&\qquad\qquad\qquad\qquad\qquad\qquad
\circ [\tilde{\sfa}_{m+1,2} \sfa_{m+2,3} \cdots \sfa_{m+q-2,q-1} \sfb_{m+q-1,q}]
\circ \sfe_{m+q,q}^{m+q,q+1} \, .
\end{split}
\end{equation}

We now wish to deduce~\eqref{eq:compDiagMaps2} from the latter.
Let us first recap a few properties of the local maps that we will use in the following:
\begin{itemize}
\item $\sfa_{i,j}$ modifies $w_{i,j}$ only, using all its nearest neighbors $w_{i-1,j}$, $w_{i,j-1}$, $w_{i+1,j}$, and $w_{i,j+1}$;
\item $\sfb_{i,j}$ modifies $w_{i,j}$ only, using $w_{i-1,j}$, $w_{i,j-1}$, and $w_{i,j+1}$;
\item $\sfc_{i,j}$ modifies $w_{i,j}$ only, using $w_{i-1,j}$ and $w_{i,j-1}$;
\item $\sfd_{i,j}^{k,l}$ modifies $w_{i,j}$ and $w_{k,l}$, using all the nearest neighbors of $w_{i,j}$ only;
\item the local maps with tilde and subscript $(i,j)$ act in the same way, except that they do \emph{not} use any entry in the $(i-1)$-th row;
\item $\sfe_{i,j}^{k,l}$ just exchanges entries $w_{i,j}$ and $w_{k,l}$.
\end{itemize}
From the above properties, one can immediately deduce several commutation relations of the local maps, which we will apply repeatedly.
Using the definitions~\eqref{eq:rho}-\eqref{eq:sigma}-\eqref{eq:tau} of the diagonal maps in terms of local maps and applying the aforementioned commutative properties, we can write the left-hand side of~\eqref{eq:compDiagMaps2} as
\[
\begin{split}
&\sigma_{m+q,q} \rho_{m+q,q+1} \tau_{m+q,q} \\
= \, & [\sfa_{m+1,1} \sfa_{m+2,2} \sfa_{m+3,3} \cdots \sfa_{m+q-1,q-1} \sfb_{m+q,q}] \circ
[\sfa_{m,1} \sfa_{m+1,2} \sfa_{m+2,3} \cdots \sfa_{m+q-1,q} \sfc_{m+q,q+1}] \\
&\circ [\sfc_{m+q,q} \sfd^{m+q,q}_{m+q-1,q-1} \cdots \sfd^{m+q,q}_{m+2,2} \sfd^{m+q,q}_{m+1,1}] \\
= \, & (\sfa_{m+1,1} \sfa_{m+2,2} \sfa_{m,1} \sfa_{m+1,2} \tilde{\sfa}_{m+1,2} \sfa_{m+2,2} \tilde{\sfa}_{m+1,1}) \\
&\circ [\tilde{\sfa}_{m+1,1} \sfa_{m+2,2} \sfa_{m+3,3} \cdots \sfa_{m+q-1,q-1} \sfb_{m+q,q}] \circ
[\tilde{\sfa}_{m+1,2} \sfa_{m+2,3} \cdots \sfa_{m+q-1,q} \sfc_{m+q,q+1}] \\
&\circ [\sfc_{m+q,q} \sfd^{m+q,q}_{m+q-1,q-1} \cdots \sfd^{m+q,q}_{m+2,2} \tilde{\sfd}^{m+q,q}_{m+1,1}] (\tilde{\sfd}^{m+q,q}_{m+1,1})^{-1} \sfd^{m+q,q}_{m+1,1} \, .
\end{split}
\]
In particular, the latter equality has been obtained by: artificially introducing $\tilde{\sfa}_{m+1,2} \tilde{\sfa}_{m+1,2}$ in between $\sfa_{m+1,2}$ and $\sfa_{m+2,3}$; moving $\sfa_{m,1} \sfa_{m+1,2} \tilde{\sfa}_{m+1,2}$ to the left; artificially introducing $\sfa_{m+2,2} \tilde{\sfa}_{m+1,1}$ and $\tilde{\sfd}^{m+q,q}_{m+1,1}$, with their respective inverses (for all these operations, recall that all $\sfa_{i,j}$'s and $\tilde{\sfa}_{i,j}$'s are involutions!).
We can now use~\eqref{eq:compDiagMapsShift}, whose left-hand side appears in the latter expression.
Setting
\[
\mathcal{A} := \sfa_{m+1,1} \sfa_{m+2,2} \sfa_{m,1} \sfa_{m+1,2} \tilde{\sfa}_{m+1,2} \sfa_{m+2,2} \tilde{\sfa}_{m+1,1}
\]
for conciseness, we thus obtain:
\[
\begin{split}
&\sigma_{m+q,q} \rho_{m+q,q+1} \tau_{m+q,q} \\
= \, &\mathcal{A} \circ
[\sfc_{m+q,q+1} \sfd^{m+q,q+1}_{m+q-1,q} \cdots \sfd^{m+q,q+1}_{m+2,3} \tilde{\sfd}^{m+q,q+1}_{m+1,2}] \circ [\tilde{\sfa}_{m+1,1} \sfa_{m+2,2} \cdots \sfa_{m+q-1,q-1} \sfc_{m+q,q}] \\
&\circ [\tilde{\sfa}_{m+1,2} \sfa_{m+2,3} \cdots \sfa_{m+q-2,q-1} \sfb_{m+q-1,q}]
\sfe_{m+q,q}^{m+q,q+1}
(\tilde{\sfd}^{m+q,q}_{m+1,1})^{-1} \sfd^{m+q,q}_{m+1,1} \, .
\end{split}
\]
Let us define $\mathcal{B}$ via
\[
\mathcal{B} \, \sfe_{m+q,q}^{m+q,q+1}
:= \sfe_{m+q,q}^{m+q,q+1}
(\tilde{\sfd}^{m+q,q}_{m+1,1})^{-1} \sfd^{m+q,q}_{m+1,1} \, .
\]
Notice that $\sfe_{m+q,q}^{m+q,q+1} \sfd^{m+q,q}_{m+1,1} \sfe_{m+q,q}^{m+q,q+1} =  \sfd^{m+q,q+1}_{m+1,1}$ by definition, and an analogous property holds when replacing the $\sfd$-maps with the corresponding $\tilde{\sfd}$.
Recalling that the exchange maps are involutions, we then have that
\[
\begin{split}
\mathcal{B}
= (\sfe_{m+q,q}^{m+q,q+1}
\tilde{\sfd}^{m+q,q}_{m+1,1}
\sfe_{m+q,q}^{m+q,q+1})^{-1}
(\sfe_{m+q,q}^{m+q,q+1}
\sfd^{m+q,q}_{m+1,1}
\sfe_{m+q,q}^{m+q,q+1})
= (\tilde{\sfd}^{m+q,q+1}_{m+1,1})^{-1} \sfd^{m+q,q+1}_{m+1,1} \, .
\end{split}
\]
Continuing to apply the commutative properties of the local maps, we obtain:
\[
\begin{split}
&\sigma_{m+q,q} \rho_{m+q,q+1} \tau_{m+q,q} \\
= \, &[\sfc_{m+q,q+1} \sfd^{m+q,q+1}_{m+q-1,q} \cdots \sfd^{m+q,q+1}_{m+3,4}] \\
&\circ \mathcal{A} \, \sfd^{m+q,q+1}_{m+2,3} \tilde{\sfd}^{m+q,q+1}_{m+1,2}
[\tilde{\sfa}_{m+1,1} \sfa_{m+2,2} \cdots \sfa_{m+q-1,q-1} \sfc_{m+q,q}]
(\tilde{\sfa}_{m+1,2} \,\mathcal{B}\, \sfa_{m+1,2} \sfa_{m,1}) \\
&\circ [\sfa_{m,1} \sfa_{m+1,2} \sfa_{m+2,3} \cdots \sfa_{m+q-2,q-1} \sfb_{m+q-1,q}]
\sfe_{m+q,q}^{m+q,q+1} \\
= \, &[\sfc_{m+q,q+1} \sfd^{m+q,q+1}_{m+q-1,q} \cdots \sfd^{m+q,q+1}_{m+2,3} \sfd^{m+q,q+1}_{m+1,2} \sfd^{m+q,q+1}_{m,1} ]
(\sfd^{m+q,q+1}_{m,1})^{-1}
(\sfd^{m+q,q+1}_{m+1,2})^{-1} \\
&\circ (\sfd^{m+q,q+1}_{m+2,3})^{-1}
\,\mathcal{A}\, \sfd^{m+q,q+1}_{m+2,3} \tilde{\sfd}^{m+q,q+1}_{m+1,2}
\tilde{\sfa}_{m+1,1} \sfa_{m+2,2} \tilde{\sfa}_{m+1,2} \,\mathcal{B}\, \sfa_{m+1,2} \sfa_{m,1} \sfa_{m+2,2} \sfa_{m+1,1} \\
&\circ [\sfa_{m+1,1} \sfa_{m+2,2} \sfa_{m+3,3} \cdots \sfa_{m+q-1,q-1} \sfc_{m+q,q}] \circ
[\sfa_{m,1} \cdots \sfa_{m+q-2,q-1} \sfb_{m+q-1,q}]
\sfe_{m+q,q}^{m+q,q+1} \, .
\end{split}
\]
For the first equality above, we have moved $\mathcal{A}$ to the right and $\mathcal{B}$ to the left; next, we have also artificially introduced the maps $\sfa_{m+1,2}$ and $\sfa_{m,1}$ and their inverses, using the fact that they are involutions.
For the second equality, we have first moved $\tilde{\sfa}_{m+1,2} \,\mathcal{B}\, \sfa_{m+1,2} \sfa_{m,1}$ to the left, and then artificially introduced $\sfd^{m+q,q+1}_{m+2,3}$, $\sfd^{m+q,q+1}_{m+1,2}$, $\sfd^{m+q,q+1}_{m,1}$, $\sfa_{m+2,2}$, and $\sfa_{m+1,1}$, with their respective inverses.
Recognizing three diagonal maps in the display above, we then deduce that
\[
\sigma_{m+q,q} \rho_{m+q,q+1} \tau_{m+q,q}
= \tau_{m+q,q+1} \,\mathcal{C}\, \rho_{m+q,q} \sigma_{m+q-1,q} \sfe_{m+q,q}^{m+q,q+1} \, ,
\]
where
\[
\begin{split}
\mathcal{C}
:= \, &(\sfd^{m+q,q+1}_{m,1})^{-1}
(\sfd^{m+q,q+1}_{m+1,2})^{-1}
(\sfd^{m+q,q+1}_{m+2,3})^{-1} \,\mathcal{A}\, \sfd^{m+q,q+1}_{m+2,3} \tilde{\sfd}^{m+q,q+1}_{m+1,2} \\
&\circ \tilde{\sfa}_{m+1,1} \sfa_{m+2,2} \tilde{\sfa}_{m+1,2} \,\mathcal{B}\, \sfa_{m+1,2} \sfa_{m,1} \sfa_{m+2,2} \sfa_{m+1,1} \, .
\end{split}
\]

Notice that the computations above hold when $q\geq 3$, i.e.\ $n\geq 6$: recall that such an assumption was made at the beginning of the inductive argument.
In fact, the local maps involved in the definition of $\mathcal{C}$ are well defined only when $q \geq 3$.
To conclude~\eqref{eq:compDiagMaps2}, it thus remains to show that the above-defined $\mathcal{C}$, which is the composition of twenty-one local maps (taking into account the definitions of $\mathcal{A}$ and $\mathcal{B}$), equals the identity.
This a straightforward, long and tedious ``by hand'' check, which can be easily carried out using a software: we have done this via an open access Mathematica package available at~\url{https://github.com/EliaBisi/LocalMaps}.

\printbibliography

\end{document}